\documentclass[reqno, 12pt]{amsart}
\usepackage{amsmath,amssymb,amsfonts,amsbsy, amsthm}
\usepackage{latexsym}
\usepackage{tcolorbox}
\usepackage{hyperref}
\usepackage[margin=2.5cm]{geometry}

\usepackage[T1]{fontenc}
\usepackage[latin9]{inputenc}
\newtheorem{theorem}{Theorem}
\usepackage[mathscr]{euscript}
\usepackage{calrsfs}

\newtheorem{thm}{Theorem}[section]
 \newtheorem{cor}[thm]{Corollary}
 
 \newtheorem{prop}[thm]{Proposition}
 \newtheorem{defn}[thm]{Definition}
 \newtheorem{rem}[thm]{Remark}

\numberwithin{equation}{section}
\newcommand{\be}{\begin{equation} \label}
\newcommand{\ee}{\end{equation}}
\newcommand{\bea}{\begin{eqnarray}\label}
\newcommand{\eea}{\end{eqnarray}}
\newcommand{\bas}{\begin{eqnarray*}}
\newcommand{\eas}{\end{eqnarray*}}
\newcommand{\bit}{\begin{itemize}}
\newcommand{\eit}{\end{itemize}}
\newcommand{\nn}{\nonumber}
\newcommand{\R}{\mathbb{R}}
\newcommand{\N}{\mathbb{N}}

\newcommand{\E}{\mathbb{E}}
 \newcommand{\Pbb}{\mathbb{P}}

\newcommand{\si}{\sigma}
\def\bge{\begin{eqnarray}}
\def\ege{\end{eqnarray}}
\def\bgee{\begin{eqnarray*}}
\def\egee{\end{eqnarray*}}

\newcommand{\la}{\lambda}

\newcommand{\no}{\nonumber}

\newcommand{\ba}{\begin{array}}
\newcommand{\ea}{\end{array}}

\def\bge{\begin{eqnarray}}
\def\bgee{\begin{eqnarray*}}
\def\ege{\end{eqnarray}}
\def\egee{\end{eqnarray*}}
\def\pl{\partial}
\def\ve{\varepsilon}

\def\R{{\mathbb R}}
\def\N{{\mathbb N}}
\def\u{z}
\def\v{v}
\def\F{S}

\DeclareMathOperator*{\esup}{ess\,sup}

\newcommand{\red}[1]{\red{#1}}
\newcommand{\cred}\red{}
\renewcommand{\red}[1]{#1}
\renewcommand{\cred}{}
\let\originallesssim\lesssim
\let\originalgtrsim\gtrsim

\DeclareRobustCommand{\lesssim}{%
  \mathrel{\mathpalette\lowersim\originallesssim}%
}
\DeclareRobustCommand{\gtrsim}{%
  \mathrel{\mathpalette\lowersim\originalgtrsim}%
}

\makeatletter
\newcommand{\lowersim}[2]{%
  \sbox\z@{$#1<$}%
  \raisebox{-\dimexpr\height-\ht\z@}{$\m@th#1#2$}%
}
\allowdisplaybreaks

\author[N.I. Kavallaris]{Nikos I. Kavallaris}
\address{
 Department of Mathematics and Computer Science, Karlstad University, Karlstad, Sweden}

\email{nikos.kavallaris@kau.se}

\author[C.V. Nikolopoulos]{Christos V. Nikolopoulos}
\address{Department of Mathematics, University of Aegean, Karlovassi, Samos, Greece}
\email{cnikolo@aegean.gr}

\author[A.N. Yannacopoulos]{Athanasios N. Yannacopoulos}
\address{Department Of Statistics, Athens University of Economics and Business,  Athens,  Greece}
\email{ayannaco@aueb.gr}

\date{\today}

\begin{document}

\title[A stochastic nonlocal difusion model]{On the impact of noise on quennching for  a nonlocal  diffusion model  driven by a mixture of Brownian and fractional Brownian motions}


\begin{abstract}
In this paper, we study a stochastic parabolic problem involving a nonlocal diffusion operator associated with nonlocal Robin-type boundary conditions. The stochastic dynamics under consideration are driven by a mixture of a classical Brownian and a fractional Brownian motions with Hurst index $H\in(\frac{1}{2}, 1).$ We first establish  local in existence result of the considered model and then explore conditions under which the resulting SPDE exhibits finite-time quenching. Using the probability distribution of perpetual integral functional of Brownian motion as well as tail estimates of fractional Brownian motion we provide analytic estimates for certain statistics of interest, such as quenching times and the corresponding quenching probabilities. The existence of  global in time solutions is also investigated and as a consequence a lower estimate of the quenching time is also derived. Our analytical results demonstrate the non-trivial impact of the considered  noise on the dynamics of the system. Next, a connection of a special case of the examined model is drawn in the context of MEMS technology. Finally, a numerical investigation of the considered model for a fractional Laplacian diffusion and Dirichlet-type boundary conditions is delivered.
\end{abstract}


\maketitle



\tableofcontents

\section{Introduction}
In this paper we consider the following nonlocal stochastic semilinear parabolic problem

\begin{eqnarray}
&&dz =\left(-\frac{1}{2 }k^2(t)\mathcal{L} z - g(x,z)\right)dt -z dN_t, \quad x\in D,\; t>0, \quad \label{model1:1}\\
&& \mathcal{N}z(x,t)+\beta(x)z(x,t)=0, \quad  x\in  D^c:=\mathbb{R}^d\setminus D,\; t>0,\label{model1:2}\\
&& 0\leq z(x,0)=z_0(x)<1, \quad x \in D,\label{model1:3}
\end{eqnarray}
where $\mathcal{L}$ is an integral (nonlocal) operator of the fractional Laplacian type and $\mathcal{N}$  is the corresponding nonlocal  Neumann-type operator whose forms are specified in Section \ref{pre}. Also, $N_{t}$ is a stochastic process which is a mixture of the Wiener process and the fractional Brownian motion with a Hurst index of $H >1/2$ and  $g$ is a nonlinear function such that $g(x,z) \to - \infty$ as $z \to 0$ whereas $g(x,z) >0$ for $\u$ large enough. A possible function $g$ with this behavior can be $g(x,z)= \lambda \zeta(x) z^{-2}-\gamma z$ for $\gamma, \lambda ,\zeta(x) >0$.  The choice of this asymptotic behaviour for $g$ in the limit as $z \to 0$ is so that the model, in its deterministic form induces quenching behaviour whereas for $z$ sufficiently large non vanishing solutions can be maintained. 

System \eqref{model1:1} -- \eqref{model1:3} is a rather general model that may be used to study the effects of spatio-temporal nonlocality and noise on quenching in nonlinear problems.
The motivation for the introduction of the integral (nonlocal) diffusion operator and the noise terms is so that we may study the combined effects of nonlocal spatio-temporal processes on the quenching behaviour of nonlinear stochastic PDEs; the considered nonlocal diffusion term modelling  spatial nonlocal effects  and the  noise term $N_{t},$ as a mixture of the Wiener process and the fractional Brownian motion,  modelling the combined effects of stochasticity and temporal nonlocal effects.
Furthermore, the choice of the noise term is so that in the limit where quenching happens $z \to 0$, the noise term will not dominate over the dynamics of the model and sweep them away under the effects of noise. On the other hand, before quenching is about to happen, i.e., when $z$ is not close to $0$, the effects of the noise term may be such as to drive the system away from the quenching regime (note that $N_{t}$ has a symmetric distribution around $0$).

 In Section \ref{cmm} we  provide a detailed example, where such a model can arise from a concrete  system of the MEMS type, and consider the general results provided in the context of the failure of MEMS systems.
The determinisitc version of \eqref{model1:1} -- \eqref{model1:3} for the case of local diffusion and for reaction terms of  the form $g(x,z)= \lambda \zeta(x) z^{-2}$ is related with applications in MEMS systems and it has been extensively studied, see \cite{EGG10, FMPS07, KMS08, KS18} and the refences therein. Nonlocal versions of those deterministic MEMS models also investigated in \cite{DKN19, GHW09, GK12, gkwy20, KLN16, M1, M2, M3, PT01}. Whilst, hyperbolic type deterministic MEMS models with both local and nonlocal nonlinearities have been considered in  \cite{F14, G10, HRL22, KLNT11, KLNT11}. The aforementioned works mainly focused on the impact of the parameter $\lambda$ and the initial data on the mathematical phenomenon of {\it quenching}, that is when $z\to 0,$ so the nonlinear term $g(x,z)$ becomes singular. Such a singular behaviour is closely related to the mechanical phenomenon of {\it touching down}, cf. \cite{KS18, JAP-DHB02, PT01}, which potentially might lead to the collapse of MEMS systems. The dynamics of stochastic systems analogous to \eqref{model1:1} -- \eqref{model1:3}, with either local or nonlocal diffusion, but with blow-up nonlinear terms $g(z)= \lambda z^{p},\; \lambda>0, \; p>1$ has been explored in \cite{AJN, DKM20, DLM, DKLM, DKLM20, DKLMT23, KY20}. There the impact of the noise, in the case of a Wiener process, a fractional Brownian motion and a mixture of them, on the phenomenon of {\it blow-up} is investigated. To this end estimates of the blow-up probability and the blow-up time are provided. The dynamics of other analogous nonlocal stochastic PDE models has been investigated in \cite{AMN20, LCC22, NK15,  XC21, XC22}. However, a stochastic model with a MEMS nonlinearity $g(z)= \lambda  z^{-2}$ and local diffusion was only first considered  in \cite{Kavallaris2016}. The author in \cite{Kavallaris2016}  considered a  multiplicative noise of the form $\pi(z) dB_t,$ for $\pi(\cdot)$ a Lipschitz function, and proved the finite-time quenching of the solution moments  for large parameter values and big enough initial data. Later on in \cite{DKN22, DNMK22} the authors explored the dynamics of stochastic MEMS models with multiplicative noise of the form $z dB_t$ and $z dB_t^H$ for Hurst index $H>\frac{1}{2}.$ The relatively simple form of the multiplicative noise term   allowed for the path-wise analysis of quenching and hence estimates of quenching probability and quenching time were obtained. This is to be contrasted with weaker notions of quenching,  which was discussed in \cite{Kavallaris2016}.

The aim of the present paper is twofold; first we examine the conditions under which quenching occurs for the stochastic problem \eqref{model1:1} -- \eqref{model1:3}. Secondly,  we obtain analytic estimates of the {\it quenching probability} as well as of the {\it quenching time},    which is a stopping time of \eqref{model1:1} -- \eqref{model1:3}. To the best of our knowledge, this is the first time that those two tasks  are considered in the context  of nonlocal diffusion SPDEs. Such considerations have their own theoretical importance in the context of singular SPDEs. Also, based on the potential connection of model \eqref{model1:1} -- \eqref{model1:3} to apllications of MEMS devices, see section \ref{cmm}, and  provided the significance of those devices in biomedical applications, such as drug delivery and diagnostics \cite{Chirkov, Nuxoll, Yager}, our analysis might offer valuable insights into their operational characteristics in the presence of a mixture noise.

The layout of the paper is as follows: in the next section we provide all the theoretical framework for the investigation of \eqref{model1:1} -- \eqref{model1:3}.
In Section \ref{lex}, we introduce a relevant to \eqref{model1:1} -- \eqref{model1:3} random PDE (RPDE) problem and define the different kind of solutions for those related problems, which are proven to be equivalent. Next we establish the local in time  existence and uniqueness of those two problems. Section \ref{upq} deals with the derivation of upper estimates for the quenching probability and the quenching time. To this end we make use of some known tail estimates for the fractional Brownian. In section \ref{leqp} we derive  lower estimates for the quenching probability by means of  the probability distribution of perpetual integral functionals of Brownian motion as well as using tail estimates of the fractional Brownian motion. Next, in section \ref{gexlb} we explore conditions under which model \eqref{model1:1} -- \eqref{model1:3} has global in time solutions. This study leads also to the derivation of lower bounds for the quenching time. We investigate the  connection of our nonlocal model with some applications in MEMS industry in section \ref{cmm}. A complementary numerical analysis of \eqref{model1:1} -- \eqref{model1:3} for the fractional Laplacian operator $(-\Delta)^{\alpha}, \; 0<\alpha<1,$ and for Dircihlet boundary conditions is conducted in section \ref{nss}, since our analytical approach is not easily applicable in that case. Then a discussion section, where our main results are highlighted, and an Appendix section, where the proof of an important auxiliary result is given,  follow.

\section{Preliminaries}\label{pre}

\subsection{The nonlocal operators $\mathcal{L}$ and $\mathcal{N}$}

By $\mathcal{L}$ we denote the nonlocal operator
\begin{eqnarray}\label{nop}
\mathcal{L}u(x):=p.v. \int_{\R^{d}}(u(x)-u(y)) k(x,y) dy, \,\,\,\,\,\, x \in \R^{d},
\end{eqnarray}
where $k : \R^{d} \times \R^{d} \setminus {\rm diag} \to [0, \infty)$ is measurable kernel satisfying
\begin{eqnarray*}
\widetilde{\Lambda}^{-1} \nu(x-y) \le k(x,y) \le \widetilde{\Lambda} \nu(x-y), \,\,\,\,\, x,y \in \R^{d},\;  \widetilde{\Lambda}>0,
\end{eqnarray*}
for some $\nu: \R^d \setminus \{0\} \to [0,\infty)$ that is the density of a symmetric L\'evy measure.
Moreover, ${\mathcal N}$ is the nonlocal analogue of the Neumann operator, defined by
\begin{eqnarray}\label{bnop}
{\mathcal N} u(y):= \int_{\Omega} (u(y)-u(x)) k(x,y) dx, \,\,\,\,\, y \in D^{c}.
\end{eqnarray}
In the special case where $k(x,y)=\nu(x-y)$ the operator $\mathcal{L}$ is translation invariant and the generator of a symmetric L\'evy process. This is the case we will focus on for this work.  For the choice $k(x,y)=\nu(x-y)= C_{d,\alpha} |x-y|^{-d-2 \alpha}$, for $\alpha \in (0,1)$ the corresponding operator $\mathcal{L}$ reduces to the fractional Laplacian, $\mathcal{L}=(-\Delta)^{\alpha}$. The usual choice of the constant $C_{d,\alpha}$ is
 $C_{d,\alpha}:= \frac{\alpha 2^{2\alpha}\Gamma\left(\frac{d}{2}+\alpha\right)}{\pi^{2\alpha+\frac{d}{2}} \Gamma(1-\alpha)}$ with $\Gamma$ being the Euler function. For this choice the operator  $\mathcal{L}$ converges (in the appropriate sense) to the standard Laplacian operator $-\Delta$ in the limit as $\alpha \to 1$.

\subsubsection{Variational formulation and relevant function spaces}

Consider the bilinear forms 
\begin{eqnarray*}
{\mathcal E}(u,v):= \int_{(D^{c} \times D^{c})^{c} } (u(x)-u(y))(v(x)-v(y)) k(x-y) dx dy,
\end{eqnarray*}
and 
\begin{eqnarray*}
Q_{\beta}(u,v):= {\mathcal E}(u,v) + \int_{D^{c}} \beta(y) u(y) v(y) dy.
\end{eqnarray*}

As shown in \cite{foghem2022general}, the bilinear form $Q_{\beta}$ is related to the nonlocal Robin problem
\bge
&&\mathcal{L} u = f,  \,\,\, \mbox{in} \,D \label{A1}\\
&&\mathcal{N} u + \beta u =0, \,\,\, \mbox{on} \,\,  D^{c}\label{A2}.
\ege

The connection between the two comes in terms of the variational formulation of \eqref{A1}-\eqref{A2} 
\begin{eqnarray}\label{WEAK-FORMULATION}
Q_{\beta}(u,v)= \int_{D} f(x) v(x) dx, \,\,\,\, \forall \,\, v \in V_{\nu}(D \mid \R^d) \cap L^{2}(D^{c} ; \beta),
\end{eqnarray}
where  
\begin{eqnarray*}
V_{\nu}(D \mid \R^d):=\left\{ u : \R^{d} \to \R \,\,\, mes. \,\, : \,\, \left . u \right\vert_{D} \in L^{2}(D), \,\,\,
 |u |_{V_{\nu}(D \mid \R^d)} < \infty\right\}
 \end{eqnarray*}
 and
\begin{eqnarray*}
| u |_{V_{\nu}(D \mid \R^d)}^2 = \int_{D}\int_{\R^d} (u(x)-u(y))^2 \nu(x-y) dx dy.
\end{eqnarray*}
This space can be turned into a Hilbert space when equipped with the norm
\begin{eqnarray*}
\| u\|_{V_{\nu}(D \mid \R^d)} = \| u \|_{L^2(D)}^2 + | u|_{V_{\nu}(D \mid \R^d)}^2.
\end{eqnarray*}
Concerning the Robin coefficient $\beta(x)$ we will assume that $\beta(x)>0$ and that $\beta \tilde{\nu}^{-1}: D^{c} \to [0, \infty)$ is essentially bounded and non-trivial (as in \cite{foghem2022general}).
The above weak formulation follows from the nonlocal Gauss-Green formula
\begin{eqnarray*}
\int_{D} \mathcal{L} u(x) v(x) dx = {\mathcal E}(u,v) - \int_{D^{c}} {\mathcal N}u(y) v(y) dy, \,\,\,\, \forall \, u,v \in C_{c}^{\infty}(\R^d),
\end{eqnarray*}
and the density of $C_{c}^{\infty}(\R^d)$ in $V_{\nu}(D \mid \R^d)$ (for $D$ such that $\partial D$ is compact and Lipschitz; see  in \cite[Theorem 2.11]{foghem2022general}).

\subsubsection{The   nonlocal Robin boundary value problem}
In \cite[Theorem 4.21]{foghem2022general}  it is shown that  the Robin eigenfunctions of $\mathcal{L}$, $\{\psi_{n} \}$ are a countable set satisfying 
\bge
&&\mathcal{L} \psi_{n} = \mu_{n} \psi_{n},   \,\,\, D \label{B1}\\
&&\mathcal{N} \psi_n + \beta \psi_n =0, \,\,\, D^{c},\label{B2}
\ege
 with $0 < \mu_1( \beta) \le \mu_{2}(\beta) \le \cdots $ such that $\{\phi_n\}$ forms an orthonormal basis of $L^{2}(D)$, and

We now state the following result which will be used in this paper (for the proof see Appendix in section  \ref{apx})

\begin{prop}\label{peip}
The first eigenvalue of the Robin problem \eqref{B1}-\eqref{B2} can be obtained by the variational formula
\begin{eqnarray}\label{Va}
\mu_1 = \min \{ Q_{\beta}(u,u) \,\, \mid \,\, \| u \|_{L^{2}(D)}=1\},
\end{eqnarray}
and the eigenfunction can be chosen to be strictly positive.
\end{prop}

\subsubsection{The semigroup generated by $\mathcal{L}$}

The operator $\mathcal{L}$  with the corresponding nonlocal boundary conditions can be shown to generate a strongly continuous semigroup on an appropriate $L^2$ space. There are two possible approaches to the choice of the required $L^2$ space. 

In the first approach (see e.g. \cite{ClausWarma} for the case of the fractional Laplacian) we may modify the boundary condition \eqref{model1:2} to
\begin{eqnarray}\label{BC-MOD}
    \mathcal{N}u+ \beta u =0, \,\,\,\, x \in \R^{d} \setminus \bar{D}
\end{eqnarray}
and then we define an extension of a function $u : D \to \R$ to a function $\bar{u} : \R^{d} \to \R$ in such a way so that the external boundary condition is automatically satisfied.  Note that the modified boundary condition carries no information as to what happens for $x \in \partial D$ but this is of little importance if we are only interested in considering $u$ as an $L^2(D)$ object since $meas(\partial D)=0$.
For the case of the Robin boundary condition studied here the corresponding extension can be defined as
\begin{equation}\label{ROBIN-EXT}
    \begin{aligned}
        \bar{u}(x) = \left\{  \begin{array}{ccc}
        u(x) &  x \in D, \\
        \frac{1}{\rho(x) + \beta(x)}  \int_{D} u(y) k(x,y) dy & x \in  \R^d \setminus D,
        \end{array} \right .
    \end{aligned}
\end{equation}
where $\rho(x):=\int_{D}k(x,y) dy$ and let $\mathcal{I}_{R}$ be the corresponding extension operator defined by $\mathcal{I}_{R}u = \bar{u}$.  It can easily be seen (see e.g.  in \cite[Lemma 3.15]{ClausWarma} for the case of the fractional Laplacian) the $\bar{u}$ satisfies the exterior boundary condition on  $\R^d \setminus \bar{D}$, for any $u$. Moreover, $u$ is well defined for $u \in L^2(D)$ (note again that $meas(\partial D)=0$). We can then define the bilinear form 
\begin{eqnarray*}
    \mathcal{E}_{R}(u,v)=Q_{\beta}(\bar{u},\bar{v})
\end{eqnarray*}
which is well defined on
\begin{eqnarray*}
    D(\mathcal{E}_{R})=\{ u \in L^{2}(D) \,\, \mid \,\, \bar{u} \in V_{\nu}(D \mid \R^d) \}.
\end{eqnarray*}
To this bilinear form we can associate an operator $\mathcal{L}_{R}$ defined as in terms of the weak formulation of the problem \eqref{A1} with the (modified) Robin boundary condition \eqref{BC-MOD} as follows: Consider any $u \in L^{2}(D)$, such that its Robin extension $\bar{u}=\mathcal{I}_{R}u$ as defined by \eqref{ROBIN-EXT} satisfies the weak formulation of \eqref{A1} with boundary condition \eqref{BC-MOD} for some function $f \in L^{2}(D)$: Then we define the operator $\mathcal{L}_{R}$ in terms of 
\begin{eqnarray*}
L^{2}(D) \ni u \mapsto \mathcal{L}_{R} u := f, \,\,\, \mbox{where} \,\,\, \bar{u}=\mathcal{I}_{R}u \,\,\, \mbox{solves} \,\,\, Q_{\beta}(\bar{u},\bar{v})=\langle f, v \rangle, \,\,\, \forall \, v \in L^2(D), \,\, \bar{v}=\mathcal{I}_{R}v.
\end{eqnarray*}
The above definition indicates that the appropriate domain for this operator will be
\begin{eqnarray*}
    D(\mathcal{L}_{R})=\{ u \in L^{2}(D) \,\, | \,\, \bar{u} \in V_{\nu}(D \mid \R^d) \,\, \exists \, f \in L^{2}(D) \,\, s.t. \,\, \bar{u} \,\, \mbox{is a weak solution of} \,\, \eqref{A1} \, \mbox{with} \, \eqref{BC-MOD}\}.
\end{eqnarray*}

It can be shown (see in \cite[Theorem 3.18]{ClausWarma} for the fractional Laplacian) that the form $\mathcal{E}_{R}$ is closed, symmetric and densely defined on $L^{2}(D)$, hence the corresponding operator $\mathcal{L}_{R}$ is a self adjoint operator generating a strongly continuous semigroup $\mathcal{T}_t^{R}=e^{-t {\mathcal L}_R}$ on $L^{2}(D)$ which moreover is positivity preserving and can be extended to a contraction semigroup on $L^{p}(D)$ for all $p \in [1,\infty]$, and   is strongly continuous for  $p \in [1, \infty)$.

In the second approach (see \cite{foghem2022general}) we do not consider directly the idea of a Robin extension but rather consider the problem on the whole of $\R^d$ and making sure that the exterior condition holds in terms of the weak formulation \eqref{WEAK-FORMULATION}. It can then be shown (see in \cite[Theorem 2.25]{foghem2022general} ) that  the bilinear form $(\mathcal{E}, V_{\nu}(D | \R^d))$ is a regular Dirichlet form on $L^2(\R^d ; \nu')$  where $\nu'$ is any of the measures
\begin{eqnarray*}
    &&\tilde{\nu}(x)= \int_{B} \min(1, \nu(x-y)) dy, \\
    &&\bar{\nu}(x)=ess \inf_{y \in B} \nu(x-y), \\
    &&\nu^{*}(x)=\nu(R(1+|x|)),
\end{eqnarray*}
where $B \subset \R^d$  is bounded and open and $R>1$  is an arbitrary fixed number. The need for the weight is so that the asymptotic behaviour of the function is controlled, since now we consider the function $u$ as a function on the whole of $\R^d$ and not just on $D$ as above. 
Here again, by the properties of Dirichlet forms the operator $\mathcal{L}$ generates an operator semigroup $\mathcal{T}_t=e^{-t \mathcal{L}}$ on $L^2(\R^d ; \nu')$, which is strongly continuous and can be extended to a contraction semigroup on $L^{p}(D ; \nu')$ for all $p \in [1, \infty]$ (by the Beurling-Deny conditions).

For the needs of this work we may consider either of the above semigroups which will be invariably denoted by $\mathcal{T}^*_t$ acting on $L^2$ (which may be either $L^{2}(D)$ or $L^2(\R^d ; \nu')$).

Moreover, we note that here we are interested in the family of operators $-\frac{1}{2}k^2(t) \mathcal{L}$.   This family of operators generates a evolution family as follows:

Define the functions
\bge
&&K(t):=\frac{1}{2}\int_0^t k^2(\tau)\,d\tau,\; A(t):=\frac{1}{2}\int_0^t a^2(\tau)\,d\tau,\label{qp8}\\
&&K(t,s):=\frac{1}{2}\int_s^t k^2(\tau)\,d\tau,\quad A(t,s):=\frac{1}{2}\int_s^t a^2(\tau)\,d\tau\label{qp8b}
\ege
We may then define the evolution family $\mathcal{T}_{t,s}$ by its action 
\begin{eqnarray*}
    \mathcal{T}_{t,s} f=\mathcal{T}^*_{K(t,s)} f,
\end{eqnarray*}
on any $f \in L^2$.  Using this evolution family we may express the solution of the nonhomogeneous system
\begin{eqnarray*}
&&    \frac{\partial u}{\partial t} = -\frac{1}{2}k^2(t) \mathcal{L} u + f, \\
&&    u(0)=\phi,
\end{eqnarray*}
in terms of
\begin{eqnarray*}
    u(t)=\mathcal{T}_t \phi + \int_{0}^{t} \mathcal{T}_{t,s} f(s) ds,
\end{eqnarray*}
where we use the simplified notation $\mathcal{T}_{t}=\mathcal{T}_{t,0}$.

\subsection{The noise term}

The noise term   $N_t$ is of mixed form and it is defined as follows:
\bge\label{cnt}
N_t:=\int_0^t a(s) dB_s+\int_0^t b(s)dB^H_s, 
\ege
for some positive functions  $a(s), b(s),$ where $B_t$ and $B^H_t$ stand for the one-dimensional, real-valued Brownian and fractional Brownian motions, defined on a stochastic basis $\{ \Omega, \,{\mathcal F}, \,{\mathcal F}_t,\,\mathbb{P} \}$ with filtration $\left({\mathcal F}_t\right)_{t\in[0,T]}$ and Hurst index $\frac{1}{2}<H<1$. 

\subsubsection{Fractional Brownian motion and its properties}

Recall that the fractional Brownian motion $(B_{t}^{H}, \,\, t \ge 0)$ is a Gaussian process with the properties
\begin{itemize}
    \item[(i)] $B_{0}^{H}=0$ a.s.\, ,
    \item[(ii)]  ${\mathbb E}[B_{t}^{H}]=0,$
    \item[(iii)]  ${\mathbb E}[B_{s}^{H}B_{t}^{H}]=\frac{1}{2}(t^{2H} + s^{2 H} - |t-s|^{2 H})=:R_{H}(t,s)$, $t,s > 0.$
\end{itemize}
The above properties uniquely characterize the fractional Brownian motion in terms of its characteristic function
\begin{eqnarray*}
    \Phi ( \lambda ; t) = {\mathbb E}[e^{i \langle \lambda, t \rangle}] =
    {\mathbb E}[e^{i \sum_{k=1}^{n} \lambda_{k}, t_{k} } ] = e^{-\frac{1}{2} \langle C_{t} \lambda ,\lambda \rangle },
\end{eqnarray*}
where  the matrix  $C_{t} : = \bigg( {\mathbb E}[B_{t_i}^{H} B_{t_j}^{H}] \bigg)_{1 \le i,j \le n}$ and $\langle  \cdot , \cdot \rangle$ denotes the standard inner product in $\R^{n}$. For the choice $H=1/2$ we recover the standard Brownian motion. The fractional Brownian motion is a self similar process (in the sense that for any $a >0$ the processes $(B_{t}^{H}, \,\, t \ge 0)$ and $(a^{-H} B_{a t}^{H}, \,\,\, t \ge 0)$ have the same distribution) and has stationary increments, which are uncorrelated for $H=1/2$, positively correlated for $H \in (1/2,1)$ and negatively correlated for $H \in (0 , 1/2)$. In this work we will focus on  positively corellated  regime $H \in (1/2,1)$ in which the process $(X_n \,\, : n \in {\mathbb N}_ := (B_{n+1}^{H}-B_{n}^{H}, \,\,\, n \in {\mathbb N})$ displays long-range dependence, which corresponds to an aggregation behaviour useful for describing clustering phenomena in various applications ranging from physical phenomena to economics (see e.g. \cite{NUALART2006, Yann2005}).  
Note that for fractional Brownian motions with $\frac{1}{2} < H <1$ it holds that $d \langle B_{t}^{H}, B_{t}^{H} \rangle =0$. On the other hand, for the standard Brownian motion $d \langle B_{t}, B_{t} \rangle =dt$.

\subsubsection{Stochastic integration with respect to fractional Brownian motion for $H >1/2$}

The fractional Brownian motion $(B_{t}^{H}, \,\,\, t \ge 0)$ is not a semimartingale for $H \ne 1/2$. For the process $(B_{t} + B_{t}^{H}, \,\,\, t \ge 0)$ though, it is interesting to note that when $H \in (3/2,1)$ it is equivalent in law to Brownian motion.

Eventhough the fractional Brownian motion is not a semimartingale a theory for stochastic integration can be developed in parallel to the corresponding theory for the Wiener process. The development of stochastic integration differs between the cases where $H<1/2$ and $H>1/2$; here we will focus in the latter case. We assume throughout the rest of the paper that $H >1/2$.
Moreover, there are different notions available for stochastic integration over the fBM, depending on the type of integrand in consideration. Various notions of stochastic integration over fBM include the pathwise integration, or the Maliavin calculus approach involving the Skorokhod integral. Different notions of stochastic integration result to integrals with different properties so we need to specify the particular concept of stochastic integration used to define the corresponding SPDE since the begining. In this paper we will use the pathwise approach.

The pathwise stochastic integral $\int_{0}^{T} u_{t} dB_{t}^{H}$ can be defined for a process $(u_t \,\,\, t \ge 0)$ with $\beta$-H\"older continuous trajectories, for $\beta > 1- H$ in the sense of Young integration as the limit of the Riemann-Stieltjes sums. Let us consider the sequence of partitions $\pi_{n}=\{0 =t_{0}^{n} <t_{1}^{n} < \cdots < t_{k_{n}}^{n}=T\}$ such that $|\pi_n| \to 0$ for $n \to \infty$. The pathwise stochastic integral is  defined as
\begin{eqnarray}\label{PATHWISE-INT}
    \int_{0}^{T} u_{t} dB^{H}_{t} = \lim_{n\to \infty} \sum_{k=1}^{k_n} u_{t_{k-1}^{n}} \left(B_{t_{k}^{n}}-B_{t_{k-1}^{n}}\right).
\end{eqnarray}
The pathwise approach follows a simple transformation rule for pointwise transformations of stochastic integrals in terms of smooth functions $F$.
In particular (see \cite[Lemma 2.7.3 page 182]{M08}) if $F \in C^2$ then
\begin{eqnarray}\label{ITO-SIMPLE-FRAC}
    F(B_{t}^{H})=F(0) + \int_{0}^{t} F'(B_{s}^{H}) dB_{s}^{H}.
\end{eqnarray}
Throughout this work, we will use the Banach space $\mathcal{B}^{\beta,2}\left(\left[0,t\right], L^2(D) \right)$, which consists of all measurable functions $u:[0,t] \rightarrow L^2(D)$ for which the norm $\Vert \cdot \Vert_{\beta,2}$ is defined, i.e., 
  $$\Vert u \Vert_{\beta,2} ^2= \left( \esup_{s \in [0,t]} \Vert u(\cdot,s)\Vert_2 \right) ^2 + \int_0^t \left( \int_0^s \frac{\Vert u(\cdot,s)- u(\cdot,r)\Vert_2}{(s-r)^{\beta+1}} dr \right)^2 ds <+ \infty, $$
 where $\Vert \cdot \Vert_2$ is the usual norm in $L^2(D)$,  cf. \cite{MN03, M08, Z}.  The requirement that $u \in \mathcal{B}^{\beta, 2} \bigl( [0,\tau],L^2(D) \bigr)$ for some  $\beta\in(1-H,1/2)$ 
 ensures that the stochastic integral with respect to $B_t^H$ in \eqref{PATHWISE-INT} exists as a generalized Stieltjes integral in the sense of \cite{Z}; see also \cite[Proposition 1]{NV06}.

As far as the mixed process $N_t$ is concerned, the stochastic integral with respect to $B_{t}^{H}$ will be considered in the above pathwise sense, whereas the stochastic integral with respect to $B_{t}$ will be considered in the It\^o sense, i.e. in the Riemann-Stieltjes sum \eqref{PATHWISE-INT} we must set $H=1/2$ and moreover, the convergence of the  Riemann sum in \eqref{PATHWISE-INT}  must be considered in the $L^{2}(\Omega, \mathbb{P})$ sense. This leads to a modification to the It\^o formula \eqref{ITO-SIMPLE-FRAC} to
\begin{eqnarray}
F(B_{t})=F(0) + \int_{0}^{t} F'(B_{s}) dB_{s} + \frac{1}{2} \int_{0}^{t} F^{\prime\prime}(B_{s}) ds.
\end{eqnarray}
Note that the last term arises because of the interpretation of the limit for the It\^o integral for $H=1/2$ in the $L^2$ sense.

The It\^o lemma can be generalized for composition of a smooth function $F$ with more complicated processes of the form $Y_{t}^{i}=\int_{0}^{t} u_{i}(s) dB_{s}^{H_i}$, for $H_{i} \in [1/2, 1)$ or linear combinations of such processes. 
The following It\^o's lemma will be used (see \cite[Theorem 2.7.3 page 184]{M08} ): Let $F : \R^{k} \to \R$ be a $C^2-$function  and let $Y^{(i)}$, $i=1, \cdots, k$ be It\^o processes driven by the mixed process $N_{t}$. Then, 
\bge\label{ITO}
F(Y^{(1)}_t, \cdots, Y^{(k)}_t)&&=F(Y^{(1)}_0, \cdots, Y^{(k)}_0) + \int_{0}^{t} \sum_{i=1}^{k} \frac{\partial F}{\partial y_{i}}(Y^{(1)}_s, \cdots, Y^{(k)}_s) dY^{(i)}_{s}\nonumber \\&&+ \frac{1}{2} \int_{0}^{t} \sum_{i,j=1}^{k} \frac{\partial ^2 F}{\partial y_i \partial y_j}(Y^{(1)}_s, \cdots, Y^{(k)}_s)  d\langle Y^{(i)}_{s},Y^{(j)}_{s}\rangle.
\ege
In the above formula the  rules  $d \langle B_{t}^{H}, B_{t}^{H} \rangle =0$, $d \langle B_{t}, B_{t} \rangle =dt$ and $d \langle B_{t}^{H}, B_{t} \rangle =0$ will be used for the calculation of the quadratic covariation processes  $d\langle Y^{(i)}_{s},Y^{(j)}_{s}\rangle$. For example, if $Y^{(i)}_t = \int_{0}^t a_{i}(s) dB_s + \int_{0}^t b_{i}(s) dB^{H}_s$, then $d\langle Y^{(i)}_{s},Y^{(j)}_{s}\rangle = a_{i}(s) a_{j}(s) ds$ since all other covariation processes vanish except for $d \langle B_{t}, B_{t} \rangle =dt$. This is because we consider the case $H>1/2$ for the fractional Brownian motion.

\subsubsection{The Malliavin derivative}

We close this section by defining the Malliavin derivative that will be used in section \ref{upq}.

To define this we will consider as $\Omega := C_{0}([0,T], \R)$ the space of continuous functions $f : [0,T] \to \R$ vanishing at $0$, equipped with the supremum norm, and set $\mathbb{P}$ the  unique probability measure on $\Omega$ such that the canonical process $( B_{t}^{H}, \,\,\, t \in [0,T])$ is the fBM with Hurst exponent $H$.

We define the square integrable Volterra kernel
$G^H$  for $H>\frac{1}{2}$ by
\bge\label{te2}
G^H(t,s):=\begin{cases}
C_H s^{1/2-H}\int_s^t (\si-s)^{H-3/2} \si^{H-1/2}\, d\si\quad \mbox{if}\quad t>s,\\0\quad\mbox{if}\quad t\leq s,
\end{cases}
\ege
where
\bgee
C_H=\left[\frac{H (2H-1)}{\mathcal{B}(2-2H, H-1/2)}\right]^{\frac{1}{2}},
\egee
and $\mathcal{B}$ stands for the usual beta function, c.f. section 5.1.3 in \cite{N06} for more details. Note that in this case $B$ and $B^H$ are dependent since $B^H$ can be represented as
\bgee
B_t^H=\int_0^t G^H(t,s) dB_s.
\egee
We then express the auto-covariance function of the fBM in terms of
\begin{eqnarray*}
    R_{H}(t,s)=\int_{0}^{\min(s,t)} G_{H}(t,r)G_{H}(s,r) dr.
\end{eqnarray*}
Let $\mathcal{S}$ be the space of step functions on $[0,T]$ and let $\mathcal{H}$ be the closure of this space with respect to the scalar product $\langle {\bf 1}_{[0,s]}, {\bf 1}_{[0,t]} \rangle_{H} =R_{H}(s,t)$.  We define the mapping ${\bf 1}_{[0,t]} \mapsto B_{t}^{H}$ on $\mathcal{S}$ and its extension to an isometry between $\mathcal{H}$ and the Gaussian space $H_{1}(B^{H})$ spanned by $B^{H}$, denoted by $\phi \mapsto B^{H}(\phi)$.


We can now define the Malliavin derivative $D$ with respect to fBM. Consider any smooth cylindrical random function of the form $F=f(B^{H}(\phi))$ where $\phi \in \mathcal{H}$, and $f \in C_{b}^{\infty}(\R)$. The Malliavin derivative $DF$ is the random variable in $\mathcal{H}$ defined by
\begin{eqnarray*}
    \langle DF, h \rangle_{\mathcal{H}} := \frac{d f}{dx}(B^{H}(\phi)) \langle \phi, h \rangle_{\mathcal{H}} = \left .\frac{d}{d\epsilon} f(B^{H}(\phi) + \epsilon \langle \phi, h \rangle_{\mathcal{H}} )\right\vert_{\epsilon=0},
\end{eqnarray*}
which can be interpreted as the generalization of the directional derivative along the possible paths of the fBM.
The derivative operator $D$ is a closable operator from $L^{p}(\Omega)$ into $L^{p}(\Omega ; \mathcal{H})$ for any $p \ge 1$. It can be used to define generalizations of Sobolev spaces. The most commonly used of this spaces is ${\mathbb D}^{1,2}$ which is the closure of the space of cylidrical smooth random variables with respect to the norm
\begin{eqnarray*}
    \| F \|_{{\mathbb D}^{1,2}} = \bigg( {\mathbb E}[|F|^2]  + {\mathbb E}[\| DF\|_{\mathcal{H}}^2] \bigg)^{1/2}.
\end{eqnarray*}
In the present setting the Malliavin derivative will be intepreted as a stochastic process $(D_{t}F, \,\,\, : \,\, t \in [0,T])$ (see \cite{N06} Section 1.2.1 and Chapter 5).

\section{Local solutions }\label{lex}

 We start by providing a precise definition of the notion of quenching time that was alluded to in the introduction. 
\begin{defn} \label{def1}
A stopping time $\tau: \Omega \to (0,\infty)$ with respect to the filtration $\{\mathcal{F}_t, t\geq 0\}$ is a quenching time of a solution $z$ of \eqref{model1:1} -- \eqref{model1:3} if 
\bgee
\lim_{t\to \tau} \inf_{x\in D} \vert z(x,t)\vert=0,  \quad\mbox{a.s. on the event}\quad \{\tau<\infty\}, 
\egee
or equivalently
\bgee
\mathbb{P}\left[\inf_{(x,t) \in D\times  [0,\tau)}\vert z(x,t)\vert >0\right]=1, \quad\mbox{a.s. on the event}\quad \{\tau<\infty\}.
\egee
We will write $\tau=+\infty$ if $\tau>t$ for all $t>0.$
\end{defn}


Next we introduce the closely related to \eqref{model1:1} -- \eqref{model1:3} RPDE problem
\bge
&&\frac{\partial \v}{\partial t}(x,t) =\left(-\frac{1}{2
}k^2(t) \mathcal{L} \v(x,t) -\frac{1}{2} a^2(t) \v(x,t) - e^{N_t} g(x, e^{-N_t}\v(x,t))\right) ,
 \quad x\in D,\; t>0,\qquad \label{RGLSP1} \\
&& \mathcal{N}\v(x,t)+\beta(x)\v(x,t)=0, \quad  x\in  D^c,\; t>0, \label{RGLSP2}\\
&& 0\leq \v(x,0)=z_0(x)<1, \quad x \in D. \label{RGLSP3}
\ege

\begin{defn}[Weak solutions] \hfill
\begin{itemize}
\item[(i)] The random field $\{ \u(x, t) \,\, : \,\, t \ge 0, \,\, x \in D \}$  with values in $L^2(D)$ is a weak solution of  \eqref{model1:1} -- \eqref{model1:3} if 
\begin{eqnarray*}
\langle \u(\cdot, t), \phi \rangle_{D}&& =\langle \u_0 , \phi\rangle_{D} + \int_{0}^{t} \bigg( - \frac{1}{2} k^2(s) \langle \u(\cdot, s), \mathcal{L}^{*}\phi \rangle_{D} - \langle g\left(\cdot,\u(\cdot, s)\right) , \phi \rangle_{D} \bigg) ds \\
&&- \int_{0}^{t} \langle \u(\cdot, s), \phi \rangle_{D} dN_{s}, \,\,\,\, \forall \, \phi \in D(\mathcal{L}^{*}),
\end{eqnarray*}
where $\mathcal{L}^{*}$  stands for the adjoint operator of $\mathcal{L}^.$
\item[(ii)] The random field  $\{ \v(x,t) \,\, : \,\, t \ge 0, \,\, x \in D \}$  with values in $L^2(D)$ is a weak solution of  \eqref{RGLSP1}--\eqref{RGLSP3}   if 
\bge\label{qp7}
&&\langle \v(\cdot, t), \phi \rangle_{D} =\langle \u_0 , \phi\rangle_{D}\no + \int_{0}^{t} \bigg(-  \frac{1}{2} k^2(s) \left\langle \v(\cdot, s), \mathcal{L}^{*}\phi \right\rangle_{D}- \frac{1}{2}a^2(s) \langle \v(\cdot, s),\phi \rangle_{D}\no\\&&  - e^{N_t}\left\langle  g\left(\cdot,e^{-N_s} \v(\cdot, s)\right) , \phi \right\rangle_{D} \bigg) ds,
 \,\,\,\, \forall \, \phi \in D(\mathcal{L}^{*}).\qquad
\ege
\end{itemize}
Note that for the present case $\mathcal{L}=\mathcal{L}^{*}.$
\end{defn}

The mild solutions for the above equations can be defined in terms of the evolution family $\mathcal{T}$ generated by the family of operators  $-\frac{1}{2}k^2(t) \mathcal{L}.$
\begin{defn}[Mild solutions] \label{MILD-DEF} \hfill
\begin{itemize}
\item[(i)] The random field $\{ \u(x, t) \,\, : \,\, t \ge 0, \,\, x \in D \}$  with values in $L^2(D)$ is a mild solution of  \eqref{model1:1} -- \eqref{model1:3} if 
\begin{eqnarray*}
\u(x, t) = [\mathcal{T}_t \u_0](x)  + \int_{0}^{t} \left(  \mathcal{T}_{t,s}\left[  -  g\left(\cdot, \u(\cdot, s)\right) \right](x) \right) ds 
+ \int_{0}^{t} \left[ \mathcal{T}_{t,s}  \u(\cdot, s) \right](x) dN_{s}, \,\,\,\, \forall \, (x,t) \in D \times [0,T].
\end{eqnarray*}
\item[(ii)] The random field  $\{ \v(x,t) \,\, : \,\, t \ge 0, \,\, x \in D \}$  with values in $L^2(D)$ is a mild solution of  \eqref{RGLSP1}--\eqref{RGLSP3} if 
\bgee
 \v(\cdot, t) =\mathcal{T}_t \v_0 + \int_{0}^{t} \mathcal{T}_{t,s}\left[ - \frac{1}{2}a^2(s)  \v(\cdot, s) - e^{N_s} g\left(\cdot, e^{-N_s} \v(\cdot, s)\right)  \right](\cdot) ds.
\egee
\end{itemize}
\end{defn}
Note that in  Definition \ref{MILD-DEF} (ii) we may rewrite the mild solution in terms of the evolution family $\widetilde{\mathcal{T}}$ generated by the family of operators ${\mathcal L}(t)=\frac{1}{2}k^2(t) \mathcal{L} - \frac{1}{2} a^{2}(t) \mathcal{I},$ where $\mathcal{I}$ represents the unit operator. This will lead to the equivalent mild form
\begin{equation}\label{MILD-MILD}
\begin{aligned}
\v(x,t)=e^{-A(t)} \mathcal{T}_t \v_0(x) 
- \int_{0}^{t} e^{N_s -A(t,s) }[ \mathcal{T}_{t,s} g\left(\cdot, e^{-N_s} \v(\cdot , s)\right)](x) ds,
\end{aligned}
\end{equation}
where  $A(t,s), A(t)$
are defined by \eqref{qp8} and \eqref{qp8b} respectively.

The following proposition allows us to reduce the SPDE problem \eqref{model1:1}--\eqref{model1:3} to the RPDE problem \eqref{RGLSP1}--\eqref{RGLSP3}.

\begin{prop}\label{SPDE-RPDE} $\u$ is a weak solution of the SPDE problem \eqref{model1:1}--\eqref{model1:3} if and only if $\v:= e^{N_t} \u$ is a weak solution of the RPDE problem \eqref{RGLSP1}--\eqref{RGLSP3}.
\end{prop}
\begin{proof}
To show that we simply need to apply It\^o's formula \eqref{ITO} to the process $Z_{t}:=e^{N_t} \langle \u(\cdot, t), \phi \rangle_{D}$,  for $\u$ being a weak solution of \eqref{model1:1}--\eqref{model1:3}. Using Definition \ref{MILD-DEF}$(i)$ we note  that we may express 
$$Z_{t}=F(Y^{(1)}_t,\bar{Y}^{(2)}_t, \hat{Y}^{(3)}_t, Y^{(4)}_t,\bar{Y}^{(5)}_t),$$ for $F(y_1,y_2,y_3,y_4,y_5)=e^{y_1+y_2} (y_3-y_4-y_5)$,  with
\begin{eqnarray*}
Y^{(1)}_t &&= \int_{0}^{t} a(s) dB_{s}, \,\,\,\,\, Y^{(4)}_{t} = \int_{0}^{t} \langle \u(\cdot, s), \phi \rangle_{D} a(s) dB_s, \\
\bar{Y}^{(2)}_{t}&&=\int_{0}^{t} b(s) dB^{H}_{s}, \,\,\,\,\, \bar{Y}^{(5)}_{t} = \int_{0}^{t} \langle \u(\cdot, s), \phi \rangle_{D} b(s) dB^{H}_s, \\
\hat{Y}^{(3)}_t&&=\langle \u_0, \phi\rangle_{D} + \int_{0}^{t} \left(-  \frac{1}{2} k^2(s) \langle \u(\cdot, s), \mathcal{L} \phi \rangle_{D} - \langle g(\u(\cdot, s)), \phi\rangle_{D} \right) ds.
\end{eqnarray*}
In order to clarify the exposition we use the following scheme in the notation: We denote the terms depending on the Wiener process by $Y^{(i)}$, the terms depending on the fractional Brownian motion by $\bar{Y}^{(i)}$ and by $\hat{Y}^{(i)}$ the process that can be expressed as a Riemann integral over $t$. Clearly, (a)  by independence there is no quadratic covariation between the processes $Y^{(i)}$ and $\bar{Y}^{(i)}$, (b) by the fact that $H >1/2$ there is no quadratic covariation between the processes $\bar{Y}^{(i)}$ and $\bar{Y}^{(j)}$ - even for $i=j$, and (c) by the standard rules of stochastic calculus there is no quadratic covariation between the processes $Y^{(i)}, \hat{Y}^{(i)}$ and $\bar{Y}^{(i)}, \hat{Y}^{(i)}$.  The only contributions in quadratic coveriation will be  from the processes $Y^{(i)}, Y^{(j)}$ (including $i=j$).  Hence, only the second derivatives $\frac{\partial^2 F}{\partial y_1^2}$ and $\frac{\partial^2 F}{\partial y_1\partial y_4}$ will contribute to the It\^o formula.  

A straightforward application of \eqref{ITO} taking into account the above comments yields that
\bge\label{ITO1}
&&F(Y^{(1)}_t,\bar{Y}^{(2)}_t, \hat{Y}^{(3)}_t, Y^{(4)}_t,\bar{Y}^{(5)}_t)=F(Y^{(1)}_0,\bar{Y}^{(2)}_0, \hat{Y}^{(3)}_0, Y^{(4)}_0,\bar{Y}^{(5)}_0)\no \\
&&+\int_{0}^{t} \left( F(Y^{(1)}_t,\bar{Y}^{(2)}_t, \hat{Y}^{(3)}_t, Y^{(4)}_t,\bar{Y}^{(5)}_t) (d Y^{(1)}_t +d\bar{Y}^{(2)}_t) + e^{ Y^{(1)}_t+\bar{Y}^{(2)}_t } d(\hat{Y}^{(3)}_t- Y^{(4)}_t-\bar{Y}^{(5)}_t) \right)\no \\
&&+\frac{1}{2} \int_{0}^{t}  \left( F(Y^{(1)}_t,\bar{Y}^{(2)}_t, \hat{Y}^{(3)}_t, Y^{(4)}_t,\bar{Y}^{(5)}_t) d\langle Y^{(1)}_t, Y^{(1)}_t  \rangle - 2  e^{Y^{(1)}_t+\bar{Y}^{(2)}_t } d\langle Y^{(1)}_t, Y^{(4)}_t  \rangle \right).
\ege
Moreover,  since $Y^{(1)}_t+\bar{Y}^{(2)}_t=N_t$, we see that
\begin{eqnarray*}
&& e^{Y^{(1)}_t+\bar{Y}^{(2)}_t } d(\hat{Y}^{(3)}_t- Y^{(4)}_t-\bar{Y}^{(5)}_t)\\&&=
 e^{N_t} \bigg ( -\frac{1}{2} k^2(s) \langle \u(\cdot, s), \mathcal{L} \phi \rangle_{D} - \langle g(\u(\cdot, s)), \phi\rangle_{D} \bigg) dt - e^{N_t} \langle \u(\cdot, t), \phi \rangle_{D} d N_t.
\end{eqnarray*}
Note that since $\u$ is a weak solution of the SPDE, 
\begin{eqnarray}\label{WEAK1}
\hat{Y}^{(3)}_t- Y^{(4)}_t-\bar{Y}^{(5)}_t = \langle \u(\cdot, t), \phi \rangle_D,
\end{eqnarray}
 it holds that 
\begin{eqnarray*}
F(Y^{(1)}_t,\bar{Y}^{(2)}_t, \hat{Y}^{(3)}_t, Y^{(4)}_t,\bar{Y}^{(5)}_t) (d Y^{(1)}_t +d\bar{Y}^{(2)}_t) = e^{N_t} \langle \u(\cdot, t) , \phi \rangle_{D} d N_t,
\end{eqnarray*}
hence, in the contribution of the first order terms in the It\^o formula, the stochastic integral over $N_t$ is eliminated to yield
\begin{eqnarray*}
&&\bigg( F(Y^{(1)}_t,\bar{Y}^{(2)}_t, \hat{Y}^{(3)}_t, Y^{(4)}_t,\bar{Y}^{(5)}_t) (d Y^{(1)}_t +d\bar{Y}^{(2)}_t) + e^{ Y^{(1)}_t+\bar{Y}^{(2)}_t } d(\hat{Y}^{(3)}_t- Y^{(4)}_t-\bar{Y}^{(5)}_t) \bigg) \\
&&=e^{N_t} \bigg (- \frac{1}{2} k^2(t) \langle \u(\cdot, t), \mathcal{L} \phi \rangle_{D} - \langle g(\u(\cdot, t)), \phi\rangle_{D} \bigg) dt.
\end{eqnarray*}
Noting that $d \langle Y_{t}^{(1)}, Y_{t}^{(1)}\rangle = a^2(t) dt$ and $d \langle Y_{t}^{(1)}, Y_{t}^{(4)}\rangle = a^2(t)  \langle \u(\cdot, t), \phi\rangle_{D} dt$, we see that (using again \eqref{WEAK1}) 
\begin{eqnarray*}
&&\frac{1}{2}  \bigg( F(Y^{(1)}_t,\bar{Y}^{(2)}_t, \hat{Y}^{(3)}_t, Y^{(4)}_t,\bar{Y}^{(5)}_t) d\langle Y^{(1)}_t, Y^{(1)}_t  \rangle - 2  e^{ Y^{(1)}_t+\bar{Y}^{(2)}_t } d\langle Y^{(1)}_t, Y^{(4)}_t  \rangle \bigg) \\
&&= - \frac{1}{2} e^{N_t}  a^2(t) \langle \u(\cdot, t), \phi\rangle_{D} dt.
\end{eqnarray*}
Substituting the above in \eqref{ITO1} and using once more \eqref{WEAK1} to express the first and second term as $e^{N_t} \langle \u(\cdot, t), \phi \rangle_{D}$ and $\langle \u_0, \phi\rangle_{D}$ respectively we obtain that
\begin{eqnarray*}
&&e^{N_t} \langle \u(\cdot, t), \phi \rangle_{D} =  \langle \u_0, \phi\rangle_{D}   \\
&&+\int_{0}^{t} \bigg( e^{N_s} \bigg ( -\frac{1}{2} k^2(s) \langle \u(\cdot, s), \mathcal{L} \phi \rangle_{D}   - \frac{1}{2}   a^2(s) \langle \u(\cdot, s), \phi\rangle_{D}-\langle g(\u(\cdot, s)), \phi\rangle_{D} \bigg)  \bigg) ds.
\end{eqnarray*}
Expressing the above in terms of $\v(x,t)=e^{N_t} \u(x,t)$ yields
\begin{eqnarray*}
&& \langle \v(\cdot, t), \phi \rangle_{D} =  \langle \u_0, \phi\rangle_{D} \\
&&+\int_{0}^{t}  \bigg (- \frac{1}{2} k^2(s) \langle \v(\cdot, s), \mathcal{L} \phi \rangle_{D}   - \frac{1}{2}   a^2(s) \langle \v(\cdot, s), \phi\rangle_{D}-\langle   e^{N_s} g(e^{-N_s} \v(\cdot, s)), \phi\rangle_{D} \bigg)   ds,
\end{eqnarray*}
which proves that $\v$ is a weak solution of the RPDE problem \eqref{RGLSP1}--\eqref{RGLSP3}.

To prove the converse, we follow a similar route, starting from a $\v$ satisfying Definition \ref{MILD-DEF} $(ii)$ and considering the process $\u =e^{-N_t} \v$ using It\^o's formula for the process $W_t=Q(Z^{(1)}_t,Z^{(2)}_t,Z^{(3)}_t)$ where 
\begin{eqnarray*}
&&Z^{(1)}_t=\int_{0}^{t} a(t) dB_t, \quad Z^{(2)}_t=\int_{0}^{t} b(t) dB_t^{H},  \\
&&Z^{(3)}_t= \langle \u_0, \phi\rangle_{D}   
+\int_{0}^{t}  \left (- \frac{1}{2} k^2(s) \langle \v(\cdot, s), \mathcal{L} \phi \rangle_{D}   - \frac{1}{2}   a^2(s) \langle \v(\cdot, s), \phi\rangle_{D}\right. \\
&& \left. -\langle   e^{N_s} g(e^{-N_s} \v(\cdot, s)), \phi\rangle_{D} \right)   ds,
\end{eqnarray*}
and $Q(z_1,z_2,z_3)=e^{-z_1-z_2} z_3$. The details are left to the reader. 
\end{proof}
Following the same reasoning as in \cite[Theorem 2.2]{DKLMT23} we can derive the equivalence between weak and mild solutions for the random problem \eqref{RGLSP1}--\eqref{RGLSP3}. Therefore the proof is omitted. 
\begin{prop}[Equivalence of weak  and mild  solutions]\label{ewms}
 $v$ is a weak solution of \eqref{RGLSP1}--\eqref{RGLSP3} in $[0,T]$ if and only $v$ is a mild solution of \eqref{RGLSP1}--\eqref{RGLSP3} in $[0,T].$
\end{prop}
\begin{prop}[Existence of local weak solutions]
There exists $T>0$ such that \eqref{model1:1}--\eqref{model1:3} has a mild solution in $L^{\infty}\left([0,T)\times D\right).$
\end{prop}
\begin{proof} By Proposition \ref{SPDE-RPDE} it suffices to show the existence of a weak local in time solution for the related RPDE problem \eqref{RGLSP1}--\eqref{RGLSP3}.

The proof is broken up in three steps.

1. We introduce the mild form for  \eqref{RGLSP1}--\eqref{RGLSP3}
\begin{equation}
\begin{aligned}
\v(x,t)=e^{-A(t)} \mathcal{T}_t \v_0(x) 
- \int_{0}^{t} e^{N_s -A(t,s) }[ \mathcal{T}_{t,s} g\left(\cdot,e^{-N_s} \v(\cdot , s) \right)](x) ds,
\end{aligned}
\end{equation}
recalling that  $\mathcal{T}$ is the evolution family generated by the operators $-\frac{1}{2}k^2(t) \mathcal{L}$ while $A(t,s), A(t)$ are defined by \eqref{qp8} and \eqref{qp8b} respectively.

By Proposition \ref{ewms} we have  that the weak form and the mild form are equivalent. Hence, showing existence for the mild form guerantees existence for the weak form as well. 

2. To prove existence we need to handle the singular behaviour of $g$ at  $\v=0$. To this end we consider the approximating sequence $(g_{n})$ defined by $g_{n}(\cdot, \v)=g\left(\cdot, \max(\v, \frac{1}{n})\right)$. Assuming the $g$ is decreasing in a neighbourhood of $0$ we see that $-g_{n}$ are bounded below by $-g\left(\cdot,\frac{1}{n}\right)$ and moreover for each $n$ the functions $g_{n}$ are locally Lipschitz. To ease notation we will denote by $C_{n}$ the bound of $|g_{n}|$ and $L_{n}(R)$ the Lipschitz constant for $g_{n}$ in $(0,R)$.

We note that for any $n$, $g_{n}(\cdot,\v)=g(\cdot,\v)$ as long as $\v > \frac{1}{n}$. Moreover, if $n < m$ then $g_{n}(\cdot,\v) \le g_{m}(\cdot,\v)$ (equiv. $-g_{n}(\cdot,\v) \ge -g_{m}(\cdot,\v)$ ). 

Using the approximation we define the  random fields $\v_{n}$ as the mild solutions for the approximate problems
\begin{equation}\label{MILD-APPROX}
\begin{aligned}
\v_n(x,t)=e^{-A(t)} \mathcal{T}_t \v_0(x) 
- \int_{0}^{t} e^{N_s -A(t,s) }\left[ \mathcal{T}_{t,s} g_n\left(\cdot, e^{-N_s} \v_n(\cdot , s)\right)\right](x) ds,
\end{aligned}
\end{equation}
and define the stopping times  $\tau_{n}=\min(\bar{\tau}_{n}, T)$, where
\begin{eqnarray*}
\bar{\tau}_{n} := \inf\left\{ t \ge 0 \,\, : \,\, \inf_{x \in D} \v_{n}(x,t) \le \frac{1}{n}\right \}.
\end{eqnarray*}

We will show that problem \eqref{MILD-APPROX}  admits a unique solution in $L^{\infty}([0,T) \times D)$ for $T$ small enough, for every $n$, so that the stopping times $\tau_{n}$ are well defined. This is due to the local Lipschitz property of $g_{n}$ and is shown in step 3 below.

Note that $\tau_{n} \le \tau_{m}$ if $n < m$. Indeed, if $\v > \frac{1}{n}$ then $g_{n}(v)=g_{m}(v),$ hence by uniqueness of the solution of \eqref{MILD-APPROX} we have that $v_{n}=v_{m}$ as long as both are  above $\frac{1}{n}$, that is as long as $t \le \tau_{n}$. This means that up to time $\tau_{n}$ neither $\v_{n}$  nor $\v_{m}$ are below $\frac{1}{n}$, hence $\v_{m}$ will go below $\frac{1}{m} < \frac{1}{n}$ at a time $\tau_{m} \ge \tau_{n}$. Therefore the sequence $(\tau_{n})$ defined above is an increasing and bounded (by $T$) sequence of stopping times.

Since up to $\tau_{n}$, it holds that $g_{n}=g,$ then $\v_{n}=\v$ in $[0,\tau_{n})$ where $\v$ is the solution of 
\eqref{MILD-MILD}. We will then construct the solution $\v$ of \eqref{MILD-MILD} via
\begin{eqnarray*}
\v(x,t)=\v_{n}(x,t), \,\,\, \forall \,\, n \ge 1, \,\, (x,t) \in D \times[0,\tau_{n}).
\end{eqnarray*}
The solution can be continued up to $\tau =\sup_{n} \tau_{n} > 0$, and is the local solution of \eqref{MILD-APPROX}.

3. We now prove the existence of a local in time mild solution  for \eqref{MILD-APPROX} in $L^{\infty}([0,T_n) \times D)$ for $T_n$ small enough.
We will use a fixed point scheme for the operator $\F_n$ defined by
\begin{equation}
\begin{aligned}
(\F_n \v)(x,t) := e^{-A(t)} \mathcal{T}_t \v_0(x) -  \int_{0}^{t} e^{N_s -A(t,s) }\left[ \mathcal{T}_{t,s} g_n\left(\cdot, e^{-N_s} \v(\cdot , s)\right)\right](x) ds,
\end{aligned}
\end{equation}
on the Banach space $X :=\{ \v : [0,T_n]\times D \to L^{\infty}(D) \,\, \mid \,\, \| \v \|_{*} <\infty \}$ where the norm $\| \cdot \|_{*}$ is defined by $\| \v \|_{*} :=\sup_{t \in [0,T_n]} \| \v(\cdot, t)\|_{\infty}$, where $T$ is a suitable time (to be determined).

We will show that the operator $\F_n$  leaves invariant and is a contraction on the subset $X_R := \{ \v \in X \,\, \mid \,\, \v \ge 0, \,\, \| \v \|_{*} \le R\}$ for a suitable $R$ hence guaranteeing the existence of a unique non negative solution.

We first consider the invariance property for $X_R$. Consider any $\v \in X_R$ and take $\F \v$. We have that
\bgee
&&\| \F_n \v \|_{*} = \sup_{ t \in [0,T] } \left\| \mathcal{T}_t \v_0 - \int_{0}^{t} e^{ N_s-A(t,s) } \left[\mathcal{T}_{t,s} g_n\left(\cdot, e^{-N_s} \v(\cdot,s)\right)\right](x) ds \right\|_{\infty} \\
&&\le \| \v_0 \|_{\infty} + \sup_{t \in [0,T]} \int_{0}^{t} e^{N_s-A(t,s)} \left\| \left[\mathcal{T}_{t,s} g_n\left(\cdot, e^{-N_s} \v(\cdot,s)\right)\right](x)  \right\|_{\infty} ds \\
&&\le \| \v_0 \|_{\infty} + \sup_{t \in [0,T]} \int_{0}^{t} e^{ N_s} \left\| \left[\mathcal{T}_{t,s} g_n\left(\cdot, e^{-N_s} \v(\cdot,s)\right)\right](x)  \right\|_{\infty} ds  \\
&&\le \| \v_0 \|_{\infty} + \sup_{t \in [0,T]} \int_{0}^{t} e^{N_s} \left\|  g_n\left(\cdot, e^{-N_s} \v(\cdot,s)\right)(\cdot)  \right\|_{\infty} ds \\
&&\le \| \v_0\|_{\infty} + T_n A_{T_n} C_{n}\\
&&\le \| \v_0\|_{\infty} + T_n A_{T_n} C_*,
\egee
 where $A_{T_n}=\sup_{t \in [0,T_n]} e^{|N_t|}$ and $C_*:=\sup_{n\in \mathbb{N}}C_n.$ Note that since $\v_n(x,0)=\v_0(x)>0,$ for $n=1,2,\dots $ then by a continuity argument we have $C_*>0.$

We will choose $R$ and $T_n$ so that
\bgee
\| \v_0\|_{\infty} + T_n A_{T_n} C_{*} < R,
\egee
and thus $X_R$ is invariant under the operator $\F_n.$

We now consider the contraction property for $\F_n.$
We start by noting that
\bge
\| \F_{n} \v_1 -\F_{n} \v_2 \|_{*} &&\le \sup_{ t \in [0,T]} \int_{0}^{t}  e^{N_s-A(t,s) } \left\| \mathcal{T}_{t,s} g_n\left(\cdot, e^{-N_s} \v_1(\cdot , s) \right) - \mathcal{T}_{t,s} g_n\left(\cdot, e^{-N_s} \v_2(\cdot , s)\right) \right\|_{\infty} ds\no \\
&&\le \sup_{ t \in [0,T]} \int_{0}^{t}  e^{N_s  } \left\| \mathcal{T}_{t,s} g_n \left(\cdot, e^{-N_s} \v_1(\cdot , s)\right) - \mathcal{T}_{t,s} g_n\left(\cdot, e^{-N_s} \v_2(\cdot , s)\right) \right\|_{\infty} ds\no \\
&&\le \sup_{ t \in [0,T]} \int_{0}^{t}  e^{N_s  } \left\|  g_n\left(\cdot, e^{-N_s} \v_1(\cdot , s)\right) -  g_n\left(\cdot, e^{-N_s} \v_2(\cdot , s)\right) \right\|_{\infty} ds,\label{EST-I}
\ege
where we first used the fact that $e^{-A(t,s)} \le 1$ and then the contractivity of the evolution family $\mathcal{T}$. To proceed further we must make use of the local Lipschitz properties of the nonlinearity $g_n$. Let $g_n$ be locally Lipschitz, satisfying the property 
\begin{eqnarray*}
| g_n(\cdot, z_1) - g_n(\cdot, z_2) | \le L_{n}(R) | z_1- z_2|, \,\,\,\, \forall \,\, z_1,z_2 \in (0,R),
\end{eqnarray*}
for some $L_{n}(R)>0$. Then as long as $|e^{-N_s}| \, \| \v_{i}(\cdot, s) \|_{\infty} < R$  we may estimate 
\begin{eqnarray*}
&&\left\|  g_n\left(\cdot, e^{-N_s} \v_1(\cdot , s)\right) -  g_n\left(\cdot,e^{-N_s} \v_2(\cdot , s)\right) \right\|_{\infty} =\sup_{x \in D} | g_n\left(\cdot, e^{-N_s} \v_1(x , s)\right) -  g_n\left(\cdot, e^{-N_s} \v_2(x, s)\right) | \\
&& \le L_{n}(R) e^{N_s} \sup_{x \in D} |\v_1(x,s) - \v_2(x,s) | =L_{n}(R) e^{-N_s} \| \v_1(\cdot, s)-\v_2(\cdot,s)\|
_{\infty}.
\end{eqnarray*}
Using the above estimate we may further estimate
\eqref{EST-I} by
\begin{eqnarray*}
\| \F_n \v_1 -\F_n \v_2 \|_{*} &&\le \sup_{ t \in [0,T]} \int_{0}^{t}  e^{N_s  } \left\|  g_n\left(\cdot, e^{-N_s} \v_1(\cdot , s) \right) -  g_n\left(\cdot, e^{-N_s} \v_2(\cdot , s)\right) \right\|_{\infty} ds \\
&&\le \sup_{ t \in [0,T]} \int_{0}^{t}  e^{N_s  } L_{n}(R)  e^{-N_s} \| v_1(\cdot, s)-v_2(\cdot,s)\|_{\infty} ds
\\
&&\le T_n L_{n}(R) \sup_{ t \in [0,T] }  \| v_1(\cdot, s)-v_2(\cdot,s)\|_{\infty} =T_n L_{n}(R) \| \v_1 - \v_2\|_{*}.
\end{eqnarray*}

Hence we conclude that
\begin{eqnarray*}
\| \F_n \v_1 -\F_n \v_2 \|_{*} \le T_n L_{n}(R) \| \v_1 - \v_2\|_{*}, \,\,\, \mbox{as long as} \,\,\,  sup_{s \in [0,T]} e^{N_s} \| \v_{i}(\cdot, s)\|_{\infty}  \le R.
\end{eqnarray*}
Choosing $T$ so that 
\begin{eqnarray}
T_n L_{n}(R)< 1,
\end{eqnarray}
it can be seen that $\F_n$ is a contraction.
\end{proof}

\section{Upper bounds for quenching time and  quenching probability}\label{upq}
The current section is focused on the derivation of upper bounds for the quenching time as well as the quenching probability for problem \eqref{model1:1}--\eqref{model1:3}. Along this section we consider the following growth condition for the nonlinearity $g,$
\bge\label{gc1}
g(x,u)\geq \eta_1 \left( \lambda \zeta(x) u^{-2}-\gamma u \right)
\ege
for some positive constant $\eta_1.$
\subsection{An upper bound for the quenching time}

Define:
\bge\label{qp8a}
\zeta_m:=\min_{x\in \bar{D}} \zeta(x)>0,\;\zeta_M:=\max_{x\in \bar{D}} \zeta(x)>0.
\ege
Consider now the following eigenvalue problem 
   \begin{eqnarray}
      && \mathcal{L} \psi= \mu \psi,\quad x\in D, \label{eig1} \\ &&\mathcal{N}\psi(x)+\beta(x) \psi=0, \quad x\in \partial D^c, \label{eig2} 
   \end{eqnarray}
then it is known, see 
Appendix, that  for the first eigenpair $(\mu_1, \psi_1)$ we have $\mu_1>0$ whilst the corresponding eigenfunction $\psi$ could be taken also positive.
 \begin{theorem}\label{thm0}
Consider the random (stopping) time 
\bge\label{qp9}
\tau^*:=\inf\left\{t>0: \int_0^t  e^{-3(\eta_1\gamma s -\mu_1 K(s)- A(s))+3N_s}\,ds\geq \frac{\langle v_0 , \psi_1 \rangle_D^3}{3 \lambda\eta_1 \zeta_m}\right\},
\ege
where the functions $K(s), A(s)$   are defined by \eqref{qp8} while the constant $\zeta_m$ is given by \eqref{qp8a}.  Then, on the event $\{\tau^*<\infty\}$ the solution $v$ of problem \eqref{RGLSP1}--\eqref{RGLSP3} and thus the solution $z$ of  \eqref{model1:1}--\eqref{model1:3} quench in finite $\tau,$ where $\tau \leq \tau^*\, \mathbb{P}-a.s.$
\end{theorem}
\begin{proof}
By choosing as a test function into the weak formulation  \eqref{qp7} the principal eigenfunction $\psi_1$ to \eqref{eig1}-\eqref{eig2} normalized as 
\bge\label{qp10}
\int_D \psi_1(x)\,dx=1,
\ege
then
\bge\label{qp11}
\langle v(\cdot,t) , \psi_1 \rangle_D&&=\langle z_0 , \psi_1 \rangle_D-\int_0^t \left(\frac{\mu_1 k^2(s)}{2}+\frac{a^2(s)}{2}\right) \langle v  , \psi_1 \rangle_D\, ds\nonumber\\
&&-\int_0^t   e^{N_s}\left\langle  g\left(\cdot,e^{-N_s} \v(\cdot, s)\right) , \psi_1 \right\rangle_{D}\,ds.
\ege

Next by \eqref{qp11} and for any $\ve>0$
\bgee
\frac{\langle v(\cdot,t+\ve) , \psi_1 \rangle_D-\langle v(\cdot,t) , \psi_1 \rangle_D}{\ve}=&&-\frac{1}{\ve}\int_t^{t+\ve} \left(\frac{\mu_1 k^2(s)}{2}+\frac{a^2(s)}{2}\right) \langle v  , \psi_1 \rangle_D\, ds\\&&-\frac{1}{\ve}\int_t^{t+\ve}   e^{N_s}\left\langle  g\left(\cdot,e^{-N_s} \v(\cdot, s)\right) , \psi_1 \right\rangle_{D}\,ds.
\egee
Letting now $\ve\to 0$
\bgee
\frac{d}{dt}\langle v(\cdot,t) , \psi_1 \rangle_D=-\left(\frac{\mu_1 k^2(t)}{2}+\frac{a^2(t)}{2}\right) \langle v(\cdot,t)  , \psi_1 \rangle_D-  e^{N_t}\left\langle  g\left(\cdot,e^{-N_s} \v(\cdot, s)\right) , \psi_1 \right\rangle_{D} . 
\egee
By virtue of \eqref{gc1} and Jensen's inequality, taking also into account \eqref{qp10}, we finally derive
\bgee
\frac{d}{dt}\langle v(\cdot,t) , \psi_1 \rangle_D \leq \left(\eta_1\gamma-\frac{\mu_1 k^2(t)}{2}-\frac{a^2(t)}{2}\right) \langle v(\cdot,t)  , \psi_1 \rangle_D- \eta_1\la \zeta_m e^{3N_t}\left\langle v(\cdot,t)  , \psi_1 \right\rangle^{-2}_D. 
\egee
Thus via comparison we deduce that $\langle v(\cdot,t) , \psi_1 \rangle_D\leq I(t)$ where $I(t)$ solves the following random initial value problem
\bgee
&& \frac{dI}{dt}=\left(\eta_1\gamma-\frac{\mu_1 k^2(t)}{2}-\frac{a^2(t)}{2}\right) I(t)- \la\eta_1 \zeta_m e^{3N_t} I^{-2}(t), \; 0<t<\tau^*, \\
&& I(0)=\langle v_0 , \psi_1 \rangle_D.
\egee
Hence
\bgee
I(t)= e^{\eta_1\gamma t-\mu_1 K(t)-A(t)}\left[\langle v_0 , \psi_1 \rangle_D^3-3\la\eta_1 \zeta_m\int_0^t e^{-3\left(\eta_1\gamma s-\mu_1 K(s)-A(s)\right)+3N_s}\,ds\right]^{1/3},
\egee
for $0<t<\tau^*,$ where $\tau^*$ is defined by \eqref{qp9}.

Therefore,  $I(t)$ hits  $0$ (quenches) in finite time on the event $\left\lbrace \tau^* < + \infty \right\rbrace.$ Since $I(t)\geq \langle v(\cdot,t) , \psi_1 \rangle_D,$ then $\tau^*$ is an upper bound $\mathbb{P}-a.s.$ for the quenching (stopping) time of $\langle v(\cdot,t) , \psi_1 \rangle_D$ and thus for  $v(x, t).$ The latter thanks to transformation $\v= e^{N_t} \u$ implies that $\tau^*$ is also an upper bound $\mathbb{P}-a.s.$ for the solution $z$ of  \eqref{model1:1}- \eqref{model1:3}.
\end{proof}
\subsection{A tail estimate leading to an upper bound of  quenching time}
In the current subesction, our main purpose is to provide a tail estimate for the upper bound $\tau^*$ defined by \eqref{qp9}, making use of a tail probability estimate for exponential estimates of fBM given in \cite[Theorem 3.1]{D18}. To this end we consider that the stochastic process is defined as follows
\bge\label{te1}
B_t^H:=\int_0^t G^H(t,s) dB_s,
\ege
where the Volterra kernel $G^H$ is given by \eqref{te2}. 

Henceforth, for simplicity, we drop the superscript $H$ from the notation of the kernel $G^H(t,s).$
\begin{thm}\label{thm1a}
Assume that $B^H$ is given by \eqref{te1} and let $$\nu(T):=\int_0^T\exp\left[-3(\gamma \eta_1 t-\mu_1 K(t)-A(t))\right]\mathbb{E}(3N_t)\,dt.$$ Then, for any $T>0$ such that 
$\frac{\langle v_0 , \psi_1 \rangle_D^3}{3 \lambda \eta_1\zeta_m}>\nu(T),$
\bgee
\mathbb{P}\left[\tau^*\leq T\right]\leq 2 \exp\left(-\frac{\ln^2 \left[\frac{\langle v_0 , \psi_1 \rangle_D^{-3}} {3 \lambda \eta_1\zeta_m\nu(T)}\right]}{2 M(T)}\right),
\egee
where 
\bge\label{te2b}
M(T):=18\int_0^T a^2(s)\,ds+36\, HT^{2H-1}\int_0^T b^2(s)\,ds.
\ege
\end{thm}
\begin{proof}
Consider the stochastic process
\bgee
X_t:=-3\left(\gamma \eta_1 t-\mu_1 K(t)-A(t)\right)+3N_t,
\egee
then thanks to \eqref{te1} we have the representation
\bge\label{te3}
X_t=-3\left(\gamma \eta_1 t-\mu_1 K(t)-A(t)\right)+3\left(\int_0^t a(s)\,dB_s+\int_0^t \int_s^t b(r) \frac{\pl}{\pl r}G(r,s)\,dr dB_s\right).
\ege
Now  from \cite[Theorem 3.1]{D18} it folows that for any $T\geq 0$ and $w>\nu(T),$ it holds
\bge\label{te4}
\mathbb{P}\left(\int_0^T e^{X_t}dt\geq w\right)\leq 2 \exp \left[-\frac{(\ln w-\ln \nu(T))^2}{2 M(T)}\right],
\ege
where $\nu(T)=\int_0^T \E\left[e^{X_t}\right]\,dt$ and $M(T)$ is such that
\bge\label{te5}
\sup_{t\in[0,T]} \int_0^T \vert D_rX_t \vert^2 dr \leq M(T)\quad \mathbb{P}-a.s.\, ,
\ege
and $D_rX_t$ stands for the Malliavin derivative of $X_t.$ 

So in order to obtain the desired tail estimate we are working towards finding an upper bound of $M(T)$ such that \eqref{te5} holds.

We first note, using the representation \eqref{te3}, that for $s<t$
\bgee
D_rX_t=3\left(a(r)+\int_r^tb(s) \frac{\pl}{\pl s}G(s,r)\,ds\right),
\egee
and thus
\bge\label{te6}
\int_0^t \vert D_rX_t \vert^2 dr\leq 18\int_0^t a^2(r)dr+18\int_0^t\left(\int_r^t b(s) \frac{\pl}{\pl s}G(s,r)\,ds\right)^2dr.
\ege
Next we are working towards the derivation of a bound of the second integral of the right-hand side of \eqref{te6}. Indeed, via Fubini's theorem and following the same calculations like in \cite[page 15]{DKLMT23} we have
\bge\label{te7}
\int_0^t\left(\int_r^t b(s) \frac{\pl}{\pl s}G(s,r)\,ds\right)^2dr
=2\int_0^t \left(\int_0^s b(s) b(s')\Psi(s,s')\,ds'\right)\,ds,
\ege
where 
\bgee
\Psi(s,s'):=\int_0^{\min\{s,s'\}}  \frac{\pl}{\pl s}G(s,r) \frac{\pl}{\pl s'}G(s',r)\,dr.
\egee
Since
\bgee
\frac{\pl}{\pl s}G(s,r) =C_H r^{1/2-H} (s-r)^{H-3/2} s^{H-1/2}
\egee
then using \cite[relation (5.7) ]{N06} for $s'<s$  we deduce
\bgee
\Psi(s,s')&&=C_H^2\, (ss')^{H-1/2} \int_0^{\min\{s,s'\}} r^{1-2H} (s-r)^{H-3/2} (s'-r)^{H-3/2}dr\\&&= H(2H-1)(s-s')^{2H-2}.
\egee
Next by virtue of \eqref{te7} and following the same calculations as in \cite[page 16]{DKLMT23} yields
\bge\label{te5a}
\int_0^t\left(\int_r^t b(s) \frac{\pl}{\pl s}G(s,r)\,ds\right)^2dr&&\leq 2H(2H-1)\int_0^t\int_0^s \vert b(s) b(s')\vert (s-s')^{2H-2} \, ds' ds\nn\\
&&\leq 2Ht^{2H-1}\int_0^t b^2(s)ds.
\ege
Then \eqref{te6} infers
\bge\label{te8}
\sup_{t\in[0,T]}\int_0^T \vert D_rX_t \vert^2 dr\leq 18\int_0^T a^2(r)dr+36HT^{2H-1}\int_0^T b^2(s)ds:=M(T).
\ege
Since by \eqref{qp9}
\bge\label{te8a}
\mathbb{P}\left[\tau^*\leq T\right]&&=\mathbb{P}\left[\int_0^T  e^{-3(\gamma \eta_1 t -\mu_1 K(t)- A(t))+3N_t}\,dt\geq \frac{\langle v_0 , \psi_1 \rangle_D^3}{3 \lambda \eta_1 \zeta_m}\right]\no\\&&=\mathbb{P}\left[\int_0^T \exp\left[X(t)\right]\,dt\geq w\right ],
\ege
then the desired result follows by virtue of \eqref{te4} and \eqref{te8}  for $w=\frac{\langle v_0 , \psi_1 \rangle_D^3}{3 \lambda \eta_1 \zeta_m}.$ 
\end{proof}
\begin{thm}\label{cvn1}
\begin{enumerate}
    \item Assume that 
    \bge\label{te9}
B_t^H= \int_0^t G(t,s)dW_s,
    \ege
where $G$ is the Volterra kernel given by \eqref{te2} and $W$ is a Brownian motion defined in the same probability space, and adapted to the same filtration as the Brownian motion $B.$ Then
\bge\label{te1a}
\mathbb{P}\left[\tau^*\leq T\right]\leq  \frac{3 \lambda \eta_1 \zeta_m \left(\int_0^T  e^{6(\mu_1 K(t)+A(t)-\gamma \eta_1 t )}\,dt+\int_0^T  e^{6 A(t)+36 Ht^{2H-1}\int_0^t b^2(s)\,ds} \,dt\right)}{\langle v_0 , \psi_1\rangle_D^3}.
\ege
\item Assume that $B_t$ and $B_t^H$ are independent then
\bge\label{te2a}
\mathbb{P}\left[\tau^*\leq T\right]\leq  3 \lambda \eta_1 \zeta_m\frac{ \int_0^T  e^{3(\mu_1 K(t)-\gamma \eta_1 t +4A(t)+3Ht^{2H-1}\int_0^tb^2(s)\,ds)}\,dt}{\langle v_0 , \psi_1 \rangle_D^3}.
\ege
\end{enumerate}
\end{thm}
\begin{proof}
\begin{enumerate}
    \item   Making use first of H\"older's and then Chebichev's inequalities  then \eqref{qp9} yields
\bge\label{te10}
\mathbb{P}\left[\tau^*\leq T\right]&&=\mathbb{P}\left[\int_0^T  e^{3(\mu_1 K(t)-\gamma \eta_1 t +\int_0^t a(s)\, dB_s)+3( A(t) +\int_0^t b(s)\, dB^H_s)}\,dt\geq \frac{\langle v_0 , \psi_1 \rangle_D^3}{3 \lambda \eta_1 \zeta_m}\right]\nn\\
&&\leq \mathbb{P}\left[\left(\int_0^T  e^{6(\mu_1 K(t)-\gamma \eta_1 t +\int_0^t a(s)\, dB_s)}\,dt\right)^{1/2}\right. \nn\\ &&\left. \times \left(\int_0^T  e^{6( A(t) +\int_0^t b(s)\, dB^H_s)}\,dt\right)^{1/2}\geq \frac{\langle v_0 , \psi_1 \rangle_D^3}{3 \lambda \eta_1 \zeta_m}\right]\nn\\
&&\leq \mathbb{P}\left[\int_0^T  e^{6(\mu_1 K(t)-\gamma \eta_1 t +\int_0^t a(s)\, dB_s)}\,dt\geq \frac{\langle v_0 , \psi_1 \rangle_D^3}{3 \lambda  \eta_1\zeta_m}\right] \nn\\ &&+\mathbb{P}\left[ \int_0^T  e^{6( A(t) +\int_0^t b(s)\, dB^H_s)}\,dt\geq \frac{\langle v_0 , \psi_1 \rangle_D^3}{3 \lambda \eta_1 \zeta_m}\right]\nn\\
&&\leq 3 \lambda \eta_1\zeta_m\left(\frac{ \E \left[\int_0^T  e^{6(\mu_1 K(t)-\gamma \delta t +\int_0^t a(s)\, dB_s)}\,dt\right]}{\langle v_0 , \psi_1 \rangle_D^3}+\frac{ \E \left[\int_0^T  e^{6( A(t) +\int_0^t b(s)\, dB^H_s)}\,dt\right]}{\langle v_0 , \psi_1 \rangle_D^3}\right)\nn\\
&&\leq \frac{3 \lambda \eta_1 \zeta_m \left(\int_0^T  e^{6(\mu_1 K(t)-\gamma \delta t +A(t))}\,dt+\int_0^T  e^{6 A(t)} \E\left[e^{6\int_0^t b(s)\, dB^H_s)}\right]\,dt\right)}{\langle v_0 , \psi_1\rangle_D^3}.
\ege
Note that in order to pass from the second to the third inequality above we have used the inequality  $ A^{1/2} B^{1/2} \le \frac{1}{2} A + \frac{1}{2} B\le  A +  B$ for the choice \bgee
A=\int_0^T  e^{6(\mu_1 K(t)-\gamma \eta_1 t +\int_0^t a(s)\, dB_s)}\,dt>0\;\mbox{and}\; B=\int_0^T  e^{6( A(t) +\int_0^t b(s)\, dB^H_s)}\,dt>0. 
\egee
Indeed, since for any $s >0$ it holds that $$\{ A^{1/2} B^{1/2} \ge s\} \subset \left\{\frac{1}{2} (A+ B) \ge s\right\}\subset \{A+ B \ge s\},$$ by the monotonicity and the subadditivity of the probability measure we pass from the second to the third inequality.
Moroever, note that in the last inequality has been used that
\bgee
\E\left[e^{\delta \int_0^t a(s)\,dB_s}\right]=\exp\left[\frac{\delta^2}{2}\int_0^t a^2(s)\,ds\right]=e^{\delta^2 A(t)}.
\egee
Furthemore thanks to \eqref{te9} and by virtue of \cite[Theorem 4.12]{Kl05} we have 
\bgee
\E\left[e^{6\int_0^t b(s)\, dB^H_s}\right]=\E\left[e^{6\int_0^t \int_s^t b(s) \frac{\pl }{\pl r}G(r,s)\,dr\, dW_s)}\right]=e^{18\int_0^t \left[\int_s^t b(s) \frac{\pl }{\pl r}G(r,s)\,dr\right]^2\, ds}.
\egee
Using now \eqref{te5a} we deduce
\bge\label{te11}
\E\left[e^{6\int_0^t b(s)\, dB^H_s}\right]\leq \exp\left(36Ht^{2H-1}\int_0^t b^2(s)ds\right).
\ege
Therefore \eqref{te10} in conjunction with \eqref{te11} implies estimate \eqref{te1a}.
    \item
In this case using again Chebichev's inequality reads
\bge\label{te10a}
\mathbb{P}\left[\tau^*\leq T\right]&&=\mathbb{P}\left[\int_0^T  e^{3(\mu_1 K(t)-\gamma \eta_1 t +\int_0^t a(s)\, dB_s)+3( A(t) +\int_0^t b(s)\, dB^H_s)}\,dt\geq \frac{\langle v_0 , \psi_1 \rangle_D^3}{3 \lambda \eta_1 \zeta_m}\right]\nn\\
&&\leq 3 \lambda \eta_1\zeta_m\frac{\mathbb{E}\left[\int_0^T  e^{3(\mu_1 K(t)-\gamma \eta_1 t +\int_0^t a(s)\, dB_s)+3( A(t) +\int_0^t b(s)\, dB^H_s)}\,dt\right]}{\langle v_0 , \psi_1 \rangle_D^3}\nn\\
&&= 3 \lambda \eta_1\zeta_m\frac{\int_0^T \mathbb{E}\left[  e^{3(\mu_1 K(t)-\gamma \eta_1 t +\int_0^t a(s)\, dB_s)}\right]\times \mathbb{E}\left[  e^{3( A(t) +\int_0^t b(s)\, dB^H_s)}\right]\,dt}{\langle v_0 , \psi_1 \rangle_D^3}\nn\\ 
&&\leq 3 \lambda \eta_1 \zeta_m\frac{ \int_0^T  e^{3(\mu_1 K(t)-\gamma \eta_1 t +4A(t)+3Ht^{2H-1}\int_0^tb^2(s)\,ds)}\,dt}{\langle v_0 , \psi_1 \rangle_D^3},
\ege
where the third equality is an immediate consequence of the independence of $B_t$ and $B_t^H,$ whilst the last inequality follows from the same calculations employed in the previous case. 
\end{enumerate}
\end{proof}
\section{Lower bounds for  quenching probability}\label{leqp}
Along this section we also assume that  the nonlinearity $g$ satisfies the growth condition \eqref{gc1}.
\subsection{General case}
\begin{thm}\label{cvn2}
Assume that \eqref{gc1} and \eqref{te9} hold true and
    \bge\label{ass1}
\int_0^t a^2(s)\,ds\sim C_1 t^{2\theta}, \quad \int_0^t b^2(s)\,ds\sim C_2 t^{2\eta}, \quad \int_0^t k^2(s)\,ds\sim C_3 t^{2\rho} \quad\mbox{as}\quad t\to \infty,
    \ege
    for some nonenegative exponents $\theta, \eta, \rho$ 
satisfying 
\bge\label{ass2}
\theta>\max\left\{\rho, H-\frac{1}{2}+\eta\right\},
\ege
and some postive constants $C_1, C_2, C_3.$

The  solution to problem \eqref{model1:1} -- \eqref{model1:3}  quenches in finite time with positive probability.

Furthermore,
\bge\label{te11a}
\mathbb{P}\left[\tau <\infty\right]\geq \mathbb{P}\left[\tau^* <\infty\right]\geq 1- \exp\left[- \frac{(m_{w}-1)^2}{2 U_{w}}\right]
\ege
recalling  $\tau^*$ is defined by \eqref{qp9}. Also
\bge\label{te12}
w=\frac{\langle v_0 , \psi_1 \rangle_D^3}{3 \lambda \eta_1\zeta_m},\quad U_{w}:=\sup_{t\geq 0} \frac{M(t)}{\left[\ln(w+1)+h(t)\right]^2},
\ege
with $h(t)=t^{2\theta}, M(t)$ given by \eqref{te2b}  and
\bge\label{te13}
m_{w}=\mathbb{E}\left[\sup_{t\geq 0} \frac{\ln \left(\int_0^t \exp{\left(-3\left(\gamma \eta_1 t-\mu_1 K(t)-A(t)-N_t\right)\right)}\,ds+1\right)+g(t)}{\ln(w+1)+h(t)}\right].
\ege
\end{thm}
\begin{proof}
Note that \eqref{te8a}, and under the assumption \eqref{gc1}, yields that
\bgee
\mathbb{P}\left[\tau^*\leq \infty\right]=\mathbb{P}\left[\int_0^{\infty} \exp\left[X(t)\right]\,dt\geq w\right],
\egee
for 
\bgee
X_t=-3\left(\gamma \eta_1 t-\mu_1 K(t)-A(t)\right)+3N_t.
\egee
Next for the estimate of the probability $\mathbb{P}\left[\int_0^{\infty} \exp\left[X(t)\right]\,dt\geq w\right]$ we will make use of \cite[Theorem 3.1]{D19}.

Indeed, using representation \eqref{te3} we infer
\bgee
\int_0^{\infty} \mathbb{E} \exp\left[X(t)\right]\,dt&&=\int_0^{\infty} \mathbb{E}\exp\left[-3\left(\gamma \eta_1 t-\frac{\mu_1}{2} \int_0^t k^2(s)\,ds -\frac{1}{2} \int_0^t a^2(s)\,ds\right)\right.\nonumber\\&& \left.+3\left(\int_0^t a(s)\,dB_s+\int_0^t \int_s^t b(r) \frac{\pl}{\pl r}G(r,s)\,dr dB_s\right)\right]\,dt\\
&&=\int_0^{\infty} \mathbb{E}\exp\left[-3\left(\gamma \eta_1 t-\frac{\mu_1}{2} \int_0^t k^2(s)\,ds -\frac{1}{2} \int_0^t a^2(s)\,ds\right)\right.\\&& \left.+3\int_0^t \left(a(s)+\int_s^t b(r) \frac{\pl}{\pl r}G(r,s)\,dr\right) dB_s\right]\,dt\\
&&=\int_0^{\infty} \exp\left[-3\left(\gamma \eta_1 t-\frac{\mu_1}{2} \int_0^t k^2(s)\,ds -\frac{1}{2} \int_0^t a^2(s)\,ds\right)\right.\\&& \left.+\frac{9}{2}\int_0^t \left(a(s)+\int_s^t b(r) \frac{\pl}{\pl r}G(r,s)\,dr\right)^2 ds\right]\,dt,
\egee
where the last inequality is an immediate consequence of \cite[Theorem 4.12]{Kl05}.

Consequeently, using \eqref{te5a} as well as \eqref{ass1} we derive
\bge\label{te14}
\int_0^{\infty} \mathbb{E} \exp\left[X(t)\right]\,dt&&\leq \int_0^{\infty} \exp\left[-3\left(\gamma \eta_1 t-\frac{\mu_1}{2} \int_0^t k^2(s)\,ds -\frac{1}{2} \int_0^t a^2(s)\,ds\right)\right.\nn\\&& \left.+9\int_0^t a^2(s) ds+18 Ht^{2H-1}\int_0^t b^2(s) ds\right]\,dt\nn\\
&&\lesssim \int_0^{\infty} \exp\left[-3\gamma \eta_1 t+\frac{3C_3\mu_1}{2} t^{2\rho}-\frac{51C_1}{2} t^{2\theta}+18C_2 Ht^{2H-1+2\eta}\right]\,dt\nn\\
&&\lesssim \int_0^{\infty} \exp\left[-3\gamma \eta_1 t-\frac{51C_1}{2} t^{2\theta}\right]\,dt<\infty,
\ege
taking also into account condition \eqref{ass2}.
In \eqref{te14} and henceforth by the notation $h_1(t)\lesssim h_2(t)$  we mean that  function $h_1(t)$ is asymptotically bounded above by $h_2(t)$ as $t\to \infty.$

Note that \eqref{te8} implies that $X_t\in \mathbb{D}^{1,2}.$ Furthermore,  \eqref{te8} in conjunction to \eqref{ass1} and \eqref{ass2} infers 
\bge\label{te14a}
\sup_{t\geq 0} \frac{\sup_{s\in [0,t]}\int_0^s \vert D_rX_s \vert^2 dr}{\left(\ln(w+1)+h(t)\right)^2}&&\leq \sup_{t\geq 0}\frac{18\int_0^t a^2(s)ds+36Ht^{2H-1}\int_0^t b^2(s)ds}{\left(\ln(w+1)+h(t)\right)^2}\nn\\&&\lesssim \sup_{t\geq 0} \frac{18C_1 t^{2\theta}+36C_2 Ht^{2(H+\eta)-1}}{\left(\ln(w+1)+h(t)\right)^2}\nn\\
&&\lesssim \sup_{t\geq 0} \frac{\widetilde{C}_1 t^{2\theta}}{\left(\ln(w+1)+h(t)\right)^2},
\ege
for some new constant $\widetilde{C}_1>0$ and fixed function $h(t).$ Choosing now $h(t)=t^{2\theta}$ then
\bgee
\lim_{t\to \infty} \frac{\widetilde{C}_1 t^{2\theta}}{\left(\ln(w+1)+h(t)\right)^2}<\infty,
\egee
and thus via \eqref{te14} we have
\bgee
\sup_{t\geq 0} \frac{\sup_{s\in [0,t]}\int_0^s \vert D_rX_s \vert^2 dr}{\left(\ln(w+1)+t^{2\theta}\right)^2}<\infty.
\egee
Now the conclusion is derived by an immediate application of \cite[Theorem 3.1]{D19} and the fact that $\tau^*$ is an upper bound of the quenching (stopping) time for the solution of problem \eqref{model1:1} -- \eqref{model1:3}.
\end{proof}
By the proof of Theorem \ref{cvn2} we easily derive the following:
\begin{cor}\label{cor1}
Consider the assumptions \eqref{ass1} and \eqref{te9} of Theorem \ref{cvn2}.
\begin{enumerate}
    \item If $\gamma=0$ and 
    \bgee
\theta>\max\left\{\rho, H-\frac{1}{2}+\eta\right\},
\egee
then the solution to problem \eqref{model1:1} -- \eqref{model1:3}  quenches in finite time with positive probability. Moreover,  \eqref{te11} holds true for
\bgee
w=\frac{\langle v_0 , \psi_1 \rangle_D^3}{3 \lambda \eta_1 \zeta_m},\quad U_{w}=\sup_{t\geq 0} \frac{M(t)}{\left[\ln(w+1)+h(t)\right]^2},
\egee
with $h(t)=t^{2\theta}$ and 
\bgee
m_{w}=\mathbb{E}\left[\sup_{t\geq 0} \frac{\ln \left(\int_0^t \exp{\left(3\left(\mu_1 K(t)+A(t)+N_t\right)\right)}\,ds+1\right)+h(t)}{\ln(w+1)+h(t)}\right].
\egee
\item If $\gamma>0$ and $0<\theta<\frac{1}{2},0<\rho<\frac{1}{2}, 0<\eta<1-H,$ then the solution to problem \eqref{model1:1} -- \eqref{model1:3}  quenches in finite time with positive probability. Also  \eqref{te11} is satisfied for
\bgee
w=\frac{\langle v_0 , \psi_1 \rangle_D^3}{3 \lambda \eta_1 \zeta_m},\quad U_{w}=\sup_{t\geq 0} \frac{M(t)}{\left[\ln(w+1)+h(t)\right]^2},
\egee
with $h(t)=t$ and 
\bgee
m_{w}=\mathbb{E}\left[\sup_{t\geq 0} \frac{\ln \left(\int_0^t \exp{\left(-3\left(\gamma \eta_1 t-\mu_1 K(t)-A(t)-N_t\right)\right)}\,ds+1\right)+h(t)}{\ln(w+1)+h(t)}\right].
\egee
\end{enumerate}
\end{cor}
\begin{cor}\label{cor2}
Consider the assumptions \eqref{ass1} and \eqref{te9} of Theorem \ref{cvn2}.
\begin{enumerate}
    \item Assume that $a(t)\equiv 0$ then for any exponents $0<\rho<\frac{1}{2}$ and $0<\eta <1-H$  the solution to problem \eqref{model1:1} -- \eqref{model1:3}  quenches in finite time with positive probability. Moreover, there holds \eqref{te11} for
\bgee
w=\frac{\langle v_0 , \psi_1 \rangle_D^3}{3 \lambda \eta_1\zeta_m},\quad U_{w}=\sup_{t\geq 0} \frac{36\, Ht^{2H-1}\int_0^t b^2(s)\,ds}{\left[\ln(w+1)+h(t)\right]^2},
\egee
with $h(t)=t^{2(H+\eta)-1}$   and
\bgee
m_{w}=\mathbb{E}\left[\sup_{t\geq 0} \frac{\ln \left(\int_0^t \exp{\left(-3\left(\gamma \eta_1 t-\mu_1 K(t)-\int_0^t b(s)\,dB_s^H\right)\right)}\,ds+1\right)+h(t)}{\ln(w+1)+h(t)}\right].
\egee
\item Assume that $b(t)\equiv 0$ then if $0<\rho<\theta$ or if  $\theta<\rho <\frac{1}{2}$ then the solution to problem \eqref{model1:1} -- \eqref{model1:3}  quenches in finite time with positive probability. In that case \eqref{te11} is satisfied for
\bgee
w=\frac{\langle v_0 , \psi_1 \rangle_D^3}{3 \lambda \eta_1\zeta_m},\quad U_{w}=\sup_{t\geq 0} \frac{18\int_0^t a^2(s)\,ds}{\left[\ln(w+1)+h(t)\right]^2},
\egee
with $h(t)=t^{2\theta}$   and
\bgee
m_{w}=\mathbb{E}\left[\sup_{t\geq 0} \frac{\ln \left(\int_0^t \exp{\left(-3\left(\gamma \eta_1 t-\mu_1 K(t)-\int_0^t a(s)\,dB_s\right)\right)}\,ds+1\right)+h(t)}{\ln(w+1)+h(t)}\right].
\egee
\end{enumerate}
\end{cor}
\subsection{An interesting special case}
In the current subscetion we are working towards the derivation of more explicit lower bound for the quenching probability $\mathbb{P}\{\tau< \infty\}.$ To this end we focus on the case $H\in(\frac{3}{4},1)$ and suppose that $B$ and $B^H$ are independent with $b(s)=C a(s)$ for any $s\geq 0$ and $C$ is a positive constant. Under those conditions $N_t=\int_0^t a(s) dM_s$ for $M_s=B_s+CB_s^H.$ It is known, see \cite[Theorem 1.7]{Ch01}, that in that specific case the stochastic process $M_t$ is equivalent in law to a Brownian motion $\widetilde{B}_t$ and so $N_t$ is equivalent $\widetilde{N}_t:=\int_0^t a(s)\,d\widetilde{B}_s.$ In addition, $(\widetilde{N}_t)_{t\geq 0}$ is a martingale and therefore a time-changed Brownian motion, in particular $\widetilde{N}_t=\widetilde{B}_{2A(t)},$ c.f. \cite{AJN, DKN22}.
\begin{thm}\label{cvn3}
Consider $H\in(\frac{3}{4},1)$ and assume that $B_t$ and $B_t^H$ are independent with $b(s)=C a(s)$ for any $s\geq 0$ where $C$ is a positive constant. In addition to growth condition \eqref{gc1} suppose also that functions $k(t)$ and $a(t)$ are positive continuous on $\mathbb{R}_{+}$ and that there exists positive constant $\Lambda$ such that
\bge\label{any1}
\frac{1}{a^2(t)} e^{3 \mu_1(K(t)+2 A(t)}\geq \Lambda, \quad\mbox{for any}\quad t>0,
\ege
and that 
\bge\label{any2}
A(t)\geq t\quad\mbox{for any}\quad t\geq 0.
\ege
\begin{enumerate}
    \item If $\gamma>\frac{1+\mu_1}{\eta_1}$ then the solution to problem \eqref{model1:1} -- \eqref{model1:3} quenches in finite time with positive probability.
    \item If $\gamma<\frac{1+\mu_1}{\eta_1}$ then the solution to problem \eqref{model1:1} -- \eqref{model1:3} quenches in finite time almost surely.
\end{enumerate}
\end{thm}
\begin{proof}
By Theorem \ref{thm0}, see \eqref{qp9}, we have
\bge\label{any3}
\mathbb{P}\left[\tau^*=\infty\right]&&=\mathbb{P}\left[\int_0^t  e^{-3(\gamma \eta_1 s -\mu_1 K(s)- A(s))+3\widetilde{N}_s}\,ds< w\quad\mbox{for all}\quad t>0\right]\nn\\&&=\mathbb{P}\left[\int_0^\infty  e^{-3(\gamma \eta_1 s -\mu_1 K(s)- A(s))+3\widetilde{N}_s}\,ds\leq w\right],
\ege
for $w=\frac{\langle v_0 , \psi_1 \rangle_D^3}{3 \lambda \eta_1 \zeta_m}.$

Under the change of variable $s_1=2 A(s)$ we derive
\bgee
\mathbb{P}\left[\tau^*=\infty\right]&&=\mathbb{P}\left[\int_0^\infty  e^{-3(\gamma \eta_1 s -\mu_1 K(s)- A(s))+3\widetilde{B}_{2A(s)}}\,ds\leq w\right]\\&&=\mathbb{P}\left[\int_0^\infty  \frac{1}{a^2\left(A^{-1}(s_1/2)\right)} e^{-3(\gamma \eta_1A^{-1}(s_1/2) -\mu_1 K(A^{-1}(s_1/2))- s_1/2)+3\widetilde{B}_{s_1}}\,ds_1\leq w\right].
\egee
Using now \eqref{any2} and the fact that $A(t)$ is a inreasing function, so that is $A^{-1}(t),$  we get $A^{-1}(s_1/2)\leq s_1/2.$ Combining that with \eqref{any1} then \eqref{any3} infers
\bge\label{any4}
\mathbb{P}\left[\tau^*=\infty\right]&&\leq \mathbb{P}\left[\int_0^\infty   e^{-3(\frac{\gamma \eta_1 s_1}{2} -\mu_1 A(A^{-1}(s_1/2))- s_1/2)+3\widetilde{B}_{s_1}}\,ds_1\leq \frac{w}{\Lambda}\right]\nn\\
&&=\mathbb{P}\left[\int_0^\infty   e^{3\frac{(1+\mu_1 -\gamma \eta_1 )s_1}{2}+3\widetilde{B}_{s_1}}\,ds_1\leq \frac{w}{\Lambda}\right].
\ege
Next, we introduce the change of variables $s_2=\left(\frac{3}{2}\right)^2s_1,$ and thus via the scaling property of Brownian motion  \eqref{any4} entails
\bge\label{any5}
\mathbb{P}\left[\tau^*=\infty\right]&&\leq \mathbb{P}\left[\frac{4}{9}\int_0^\infty   e^{3\widetilde{B}_{\frac{4}{9}s_2}+\frac{3}{2}\frac{(1+\mu_1 -\gamma \eta_1 )4s_2}{9}}\,ds_2\leq \frac{w}{\Lambda}\right]\nn\\
&&=\mathbb{P}\left[\int_0^\infty   e^{2\widetilde{B}_{s_2}+\frac{2}{3}(1+\mu_1 -\gamma \eta_1 )s_2}\,ds_2\leq \frac{9w}{4\Lambda}\right]\nn\\
&&=\mathbb{P}\left[\int_0^\infty   e^{2\widetilde{B}^{(\nu)}_{s}}\,ds\leq \frac{9w}{4\Lambda}\right],
\ege
where $\nu:=\frac{1+\mu_1-\gamma \eta_1}{3}$ and $\widetilde{B}^{(\nu)}_{s}:=\widetilde{B}_{s}+\nu s.$

Next we distinguish the following cases:
\begin{enumerate}
    \item Consider that $\nu<0,$ that is $\gamma>\frac{1+\mu_1}{\eta_1}.$ Then 
 it is known that
  \begin{eqnarray}\int_0^{\infty} e^{2\widetilde{B}^{(\nu)}_{s}} ds\stackrel{\text{Law}}{=} \frac{1}{2 Z_{-\nu}},
  \nonumber
    \end{eqnarray}
 cf. see\cite[ Chapter 6, Corollary 1.2]{YOR},   where $Z_{-\nu}$ is a random variable with law $\Gamma(-\nu),$ i.e. $$\Pbb\left[Z_{-\nu}\in dy\right]=\frac{1}{\Gamma(-\nu)} e^{-y} y^{-\nu-1}\,dy,$$
 where $\Gamma(\cdot)$ is the complete gamma function, cf. \cite{AS}.

Hence \eqref{any5} entails (see also  in \cite[formula 1.104(1) page 264]{BS02})
\bge\label{any5a}
\mathbb{P}\left[\tau^*=\infty\right]&&\leq \Pbb\left[\frac{1}{2 Z_{-\nu}}\leq \frac{9w}{4\Lambda}\right]\nn\\
     &&=\Pbb\left[ Z_{-\nu}\geq \frac{2\Lambda}{9w}\right]\nn\\
     &&=1-\Pbb\left[ Z_{-\nu}\leq \frac{2\Lambda}{9w}\right]\nn\\
     &&=1-\frac{1}{\Gamma(-\nu)}  \int_0 ^{\widetilde{\Lambda}} e^{-y} y^{-\nu-1}\,dy,
\ege
where $\widetilde{\Lambda}:=\frac{2\Lambda}{9w}.$
Therefore
\bgee
 \mathbb{P}\left[\tau^*=\infty\right]\geq \Pbb\left[ Z_{-\nu}\leq \widetilde{\Lambda}\right]=\frac{1}{\Gamma(-\nu)}  \int_0 ^{\widetilde{\Lambda}} e^{-y} y^{-\nu-1}\,dy.
\egee
Now since $\tau<\tau^*$ we have that
\bge\label{LSPgProp}
 \Pbb\left[ \tau <+\infty \right]&&\geq \frac{1}{\Gamma(-\nu)}  \int_0 ^{\widetilde{\Lambda}} e^{-y} y^{-\nu-1}\,dy\nn\\
 &&=\int_0 ^{\frac{2 \lambda \eta_1 \Lambda  \zeta_m} {3\langle v_0 , \psi_1 \rangle_D^3}} \frac{e^{-y}y^{-\frac{\mu_1+7-\gamma}{3}}}{ \Gamma \left( - \frac{\mu_1 +4-\gamma}{3} \right)} \, dy.
 \ege
\item Assume that $\nu>0,$ and so $\gamma<\frac{1+\mu_1}{\eta_1}.$
Then, by virtue of the law of the iterated logarithm for the Brownian motion  $\widetilde{B}_t,$  cf.  \cite[Theorem 2.3]{A95} and \cite[Theorem 9.23]{KS91}, that is
\bge&&\lim \inf_{t \to + \infty} \frac{\widetilde{B}_t}{t^{1/2} \sqrt{2 \log ( \log t)}}=-1, \quad  \mathbb{P}- a.s.\; ,\label{psk4}\\
&&\mbox{and}\nn\\
&&\lim \sup_{t \to +\infty} \frac{\widetilde{B}_t}{t^{1/2} \sqrt{2 \log ( \log t)}}=+1, \quad  \mathbb{P}- a.s.\; ,\label{psk5}
\ege
we deduce that for any sequence $t_n\to +\infty$
\bgee
\widetilde{B}_{t_n} \sim \alpha_n t_n^{1/2} \sqrt{2 \log (\log t_n)},
\egee
with $\alpha_n\in[-1,1],$ and thus
\bgee
\int_0^\infty   e^{3\frac{(1+\mu_1 -\gamma \eta_1 )s_1}{2}+3\widetilde{B}_{s_1}}\,ds_1=+\infty.
\egee
The latter, due to \eqref{any4}, infers
\begin{eqnarray}
\mathbb{P} \left[\tau^*=+ \infty \right]=0,
\end{eqnarray}
and hence
\begin{eqnarray}
\mathbb{P}\left[\tau^* <+\infty \right]= 1-\mathbb{P}[\tau^*=+\infty]=1-0=1.
\end{eqnarray}
Now since $\tau<\tau^*$ then $\mathbb{P}\left[\tau <+\infty \right]=1$ as well and thus the solution to problem \eqref{model1:1} -- \eqref{model1:3} quenches in finite time almost surely.
\end{enumerate}
\end{proof}
\begin{rem}
Note that conditions \eqref{any1} is satisfied, whilst \eqref{any2} is not needed anymore,  for constant functions $a(t)\equiv a $ and $k(t)\equiv k$ and thus Theorem \ref{cvn3} is valid in that special case. Indeed, for $\gamma>0$ and following the same steps as in the proof of the first part of Theorem \ref{cvn3}  we obtain that the solution of   problem \eqref{model1:1} -- \eqref{model1:3}
quenches in finite time with positive probability
\bgee
 \Pbb\left[ \tau <+\infty \right]\geq \int_0 ^{\frac{4 \lambda \eta_1 \zeta_m} {3a^2\langle v_0 , \psi_1 \rangle_D^3}} \frac{e^{-y}y^{-\frac{9\gamma}{8a^2}}}{ \Gamma \left( - \frac{9\gamma}{8a^2} \right)} \, dy.
 \egee
\end{rem}

\section{Global existence--Lower bound of quenching time}\label{gexlb}
In this section we consider the following growth condition for the nonlinearity $g,$
\bge\label{gc2}
g(x,u)\leq \eta_2 \left( \lambda \zeta(x) u^{-2}-\gamma u \right).
\ege
 Then thanks to \eqref{MILD-MILD} any mild solution of problem \eqref{RGLSP1}--\eqref{RGLSP3} satisfies the integral inequality   
\bge\label{psk0}
\v(x,t)\geq e^{\gamma \eta_2 t-A(t)} \mathcal{T}_t \u_0(x) 
- \la \eta_2\int_{0}^{t} e^{\gamma\eta_2 (t-s)-A(t,s)} e^{3N_s}\left[ \mathcal{T}_{t-s} \left(\zeta(\cdot) \v^{-2}(\cdot , s)\right)\right](x) ds.
\ege
Notably, the weak and mild solutions are equivalent for problem \eqref{RGLSP1}--\eqref{RGLSP3}, see  Proposition \ref{ewms}, therefore in the current subection we will use the concept of a mild solution to derive lower bounds for  the quenching (stopping) time for $v$ and $z.$
 To this end, in the sequel we follow the same approach as in \cite{DNMK22}  (see also \cite{DLM}) to derive such a lower bound.  

Consider the stochastic process
\begin{eqnarray}\label{ik10}
0<\mathcal{G}(t):=\left[1-4\lambda \eta_2\zeta_M \int_0^t e^{3  N_r}\mu^{-3}(r)\, dr\right]^{1/4}<1,
\end{eqnarray}
where $\zeta_M$ is defined by \eqref{qp8a} and 
\bge\label{sk2}
\mu(t):=e^{\eta_2\gamma t-A(t)}\inf_{x\in D} \mathcal{T}_t v_0(x)>0,
\ege
due to the initial condition. Its stopping time is defined by
\begin{eqnarray}\label{ik11}
\tau_*:=\inf\left\{t>0: \int_0^t e^{3  N_r} \mu^{-3}(r)\, dr\geq \frac{1}{4 \lambda\eta_2\zeta_M} \right\}.
\end{eqnarray}
Our first result towards the derivation of this lower bound is the following:
\begin{thm}\label{psk1}
Let $\tau_{*}$ be the stopping time given by \eqref{ik11}). Consider the stochastic process $\mathcal{G}(t)$ defined by \eqref{sk2} for any $t \in [0,\tau_{*}]$. Then, problem \eqref{model1:1}- \eqref{model1:3} admits a solution $z$ in $[0,\tau_{*}]$ that satisfies 
\begin{eqnarray}\label{ik21a}
0<\mathcal{F}(x,t)\mathcal{G}(t) \leq z(x,t)\leq 1, 
\end{eqnarray}
where $\mathcal{F}(x,t):=e^{\gamma \eta_2 t-A(t)-N_t} \mathcal{T}_t v_0(x).$
\end{thm}
\begin{proof}
Note that $\mathcal{G}(0)=1.$ By differentiating \eqref{ik10}, we obtain
\begin{eqnarray*}
\mathcal{G}'(t)=-\lambda \eta_2 \zeta_M e^{3 N_t}\mu^{-3}(t)\mathcal{G}^{-3}(t),
\end{eqnarray*}
and thus
\begin{eqnarray*}
\mathcal{G}(t)=
1-\lambda \eta_2\zeta_M\int_0^t  e^{3 N_r} \mu^{-3}(r) \mathcal{G}^{-3}(r)\,dr.
\end{eqnarray*}
Set
\begin{eqnarray}\label{ik3}
\mathcal{R}(U)(x,t):=e^{\gamma \eta_2 t-A(t)} \mathcal{T}_t v_0(x)
-\lambda \eta_2\int_0^t e^{\gamma \eta_2 (t-r)-A(t,r)} e^{3 N_r}  \mathcal{T}_{t-r}\left(\zeta(\cdot)U^{-2}(\cdot,r)\right)(x)\,dr, \quad
\end{eqnarray}
for $x\in D,\, t\geq 0,$ where $(t,x)\mapsto U(x,t)$ is any nonnegative function such that $U(\cdot,t)\in C_0(D), t\geq 0,$ and 
\begin{equation}\label{ik4}
\begin{aligned}
\mathcal{G}_1(x,t):=e^{\gamma \eta_2 t-A(t)}\mathcal{T}_t v_0(x)\mathcal{G}(t)\leq \vert U(x,t) \vert\leq \mathcal{G}_2(x,t):=e^{\gamma \eta_2 t-A(t)}\mathcal{T}_t v_0(x) ,\\ \; x\in D, \; \;    0\leq t< \tau_*\leq \infty.
\end{aligned}
\end{equation}
In the following we denote $U_t(x):=U(x,t)$ for simplicity and without any confussion. 

An immediate consequence of \eqref{ik3} is $\mathcal{R}(U)(x,t)\leq e^{\gamma \eta_2 t-A(t)}\mathcal{T}_t v_0(x).$ The latter in conjunction with \eqref{ik4} gives
\begin{eqnarray}\label{ik5}
\mathcal{R}(U)(x,t)&&=e^{\gamma \eta_2 t-A(t)} \mathcal{T}_t v_0(x)
-\lambda \eta_2 \int_0^t e^{\gamma \eta_2 (t-r)-A(t,r)} e^{3 N_r}  \mathcal{T}_{t,r}\left(\zeta(\cdot) U_r^{-3}(\cdot)  U_r(\cdot)\right)(x)\,dr\nonumber\\
&&\geq e^{\gamma \eta_2 t-A(t)}  \mathcal{T}_t v_0(x)
-\lambda \eta_2 \zeta_M\int_0^t e^{\gamma \eta_2 (t-r)-A(t,r)} e^{3 N_r}  \mathcal{T}_{t,r}\left( U_r^{-3}(\cdot) U_r (\cdot)\right)(x)\,dr\nonumber\\
&&\geq e^{\gamma \eta_2 t-A(t)} \mathcal{T}_t v_0(x)
-\lambda \eta_2 \zeta_M\int_0^t e^{\gamma \eta_2 (t-r)-A(t,r)} e^{3 N_r}  \mathcal{T}_{t,r}\left(\mathcal{G}_1^{-3}(\cdot, r) U_r(\cdot)\right)(x)\,dr\nonumber\\
&&\geq e^{\gamma \eta_2 t-A(t)} \mathcal{T}_t v_0(x)
-\lambda \eta_2 \zeta_M\int_0^t e^{\gamma (t-r)-A(t,r)} e^{3 N_r} \mathcal{G}^{-3}(r) \mu^{-3}(r)  \mathcal{T}_{t,r}\left(U_r(\cdot) \right)(x)\,dr.\qquad \qquad
\end{eqnarray}
Next, using once more  \eqref{ik4} and applying the semigroup property \cite[Definition 2.3, page 106]{P83}  into the last inequality of \eqref{ik5}  
reads
\bge\label{psk00}
\mathcal{R}(U)(x,t)&&\geq e^{\gamma \eta_2 t-A(t)} \mathcal{T}_t v_0(x)
-\lambda \eta_2 \zeta_M\int_0^t e^{\gamma \eta_2 (t-r)-A(t,r)} e^{3 N_r} \mathcal{G}^{-3}(r) \mu^{-3}(r)  \mathcal{T}_{t,r}\left( U_r(\cdot)\right)(x)\,dr\nonumber\\
&&\geq e^{\gamma \eta_2 t-A(t)} \mathcal{T}_t v_0(x)
-\lambda \eta_2 \zeta_M\int_0^t e^{\gamma \eta_2 (t-r)-A(t,r)} e^{3 N_r} \mathcal{G}^{-3}(r) \mu^{-3}(r)  \mathcal{T}_{t,r}\left( \mathcal{G}_2(\cdot,r)\right)(x)\,dr\nonumber\\
&&=e^{\gamma \eta_2 t-A(t)} \mathcal{T}_t v_0(x)
-\lambda \eta_2 \zeta_M\int_0^t e^{\gamma \eta_2 t-A(t)} e^{3 N_r} \mathcal{G}^{-3}(r) \mu^{-3}(r)  \mathcal{T}_{t,r}\left( \mathcal{T}_r v_0(x)\right)(x)\,dr\nonumber\\
&&= e^{\gamma \eta_2 t-A(t)} \mathcal{T}_t v_0(x)\left[
1-\lambda \eta_2 \zeta_M\int_0^t   e^{3 N_r} \mathcal{G}^{-3}(r)\mu^{-3}(r)\,dr\right]\nonumber\\
&&= e^{\gamma \eta_2 t-A(t)} \mathcal{T}_t v_0(x) \mathcal{G}(t).
\ege
Accordingly,
\begin{eqnarray}\label{ik12}
e^{\gamma \eta_2 t-A(t)}\mathcal{T}_t v_0(x)\mathcal{G}(t)\leq \mathcal{R}(U)(x,t)\leq e^{\gamma\eta_2 t-A(t)}\mathcal{T}_t v_0(x) , \quad 0\leq t< \tau^*, \, x\in D.
\end{eqnarray}
Next, we consider the iteration scheme
\begin{eqnarray*}
v_0(x,t):=e^{\gamma \eta_2 t-A(t)} \mathcal{T}_t v_0(x)\quad \mbox{and}\quad  v_{n+1}(x,t)=\mathcal{R}(v_{n})(x,t),\quad n=0,1,2,\dots.
\end{eqnarray*}
Due to \eqref{ik12}, $v_0(x,t)\geq \mathcal{R}(v_0)(x,t)=v_1(x,t).$ If we assume that $v_{n-1}(x,t)\geq v_{n}(x,t)$ for some $n\geq 1$ and for every $x\in D$ and $t\geq 0,$ then since $\zeta(x)>0$
\begin{eqnarray*}
v_{n+1}(x,t)&&=\mathcal{R}(v_{n})(x,t)\\
&&=e^{\gamma \eta_2 t-A(t)} \mathcal{T}_t v_0(x)
-\lambda \eta_2 \int_0^t e^{\gamma \eta_2 (t-r)- A(t,r)} e^{3 N_r}  \mathcal{T}_{t-r}\left({\zeta v_{n}^{-2}}\right)(x)\,dr\\
&&\leq e^{\gamma \eta_2 t-A(t)} \mathcal{T}_t v_0(x)
-\lambda \eta_2 \int_0^t e^{\gamma \eta_2 (t-r)-A(t,r)} e^{3 N_r}  \mathcal{T}_{t,r}\left(\zeta v_{n-1}^{-2}\right)(x)\,dr\\
&&=\mathcal{R}( v_{n-1})(x,t)\\
&&=v_{n}(x,t).
\end{eqnarray*}
The latter implies by induction that $\{v_n\}_{n=0}^{\infty}$ is a decreasing sequence of nonnegative functions. Therefore, the limit
\begin{eqnarray*}
\widetilde{v}(x,t)=\lim_{n\to \infty}v_n(x,t),
\end{eqnarray*}
exists for every $x\in D$ and all $0\leq t< \tau_*.$

Consequently, using the version of the monotone convergence theorem for decreasing functions reads
\begin{eqnarray*}
\widetilde{v}(x,t)=\mathcal{R}(\widetilde{v})(x,t)\quad\mbox{for}\quad x\in D\quad\mbox{and}\quad 0\leq t< \tau_*,
\end{eqnarray*}
and hence $\widetilde{v}(x,t)$ coincides with the unique mild solution $v_1(x,t)$ of the following random problem 
\bgee
&&\frac{\partial v_1}{\partial t}(x,t) =-\frac{1}{2
}k^2(t) \mathcal{L} v_1(x,t)+\left(\gamma \eta_2-\frac{1}{2}a^2(t)\right) v_1(x,t) - \lambda \eta_2 e^{3N_t}\frac{\zeta(x)}{v^2_1(x,t)}  ,
 \quad x\in D,\; t>0, \\
&& \mathcal{N}v_1(x,t)+\beta(x)v_1(x,t)=0, \quad  x\in  D^c,\; t>0, \\
&& 0\leq v_1(x,0)=z_0(x)<1, \quad x \in D. 
\egee
which is given as the following integral representation
\bgee
v_1(x,t)= e^{\gamma \eta_2 t-A(t)} \mathcal{T}_t \u_0(x) 
- \la \eta_2\int_{0}^{t} e^{\gamma\eta_2 (t-s)-A(t,s)} e^{3N_s}\left[ \mathcal{T}_{t,s} \left(\zeta(\cdot) v_1^{-2}(\cdot , s)\right)\right](x) ds.
\egee
Furthermore, by virtue of \eqref{psk0}, \eqref{ik10}, \eqref{psk00}   then the comparison principle, cf. \cite{BM20}, yields
\begin{eqnarray}\label{ik21}
0<e^{\gamma \eta_1 t-A(t)} \mathcal{T}_t v_0(x)\mathcal{G}(t) \leq v_1(x,t)\leq v(x,t)\leq  1,\; x\in D, \; 0\leq t< \tau_*\leq \infty.\qquad
\end{eqnarray}
 Now, \eqref{ik21a} is an immediate consequence of \eqref{ik21} using also that $v(x,t)=z(x,t)e^{ N_t}.$
\end{proof}
\begin{rem}\label{sk1}
By the proof of Theorem  \ref{psk1}, cf. \eqref{ik21}, we conclude that the quenching time $\tau$ for the solution $z$ of  problem \eqref{model1:1}- \eqref{model1:3} is bounded below by the random variable $\tau_*$ defined by \eqref{ik11}, that is $\tau^*\leq \tau.$
 \end{rem}
Next using Theorem  \ref{psk1},we  provide a condition under which problem \eqref{model1:1}- \eqref{model1:3} has a  global in time solution alsmost surely.
\begin{cor}\label{sk3}
Consider initial data $0<z_0(x)=\v_0(x)\leq 1$ for problem \eqref{model1:1}- \eqref{model1:3} satisfying the condition
\begin{eqnarray}\label{ik22}
\int_0^{\infty} e^{3 N_r}\mu^{-3}(r)\, dr<\frac{1}{4 \lambda \eta_2 \zeta_M}.
\end{eqnarray}
Then problem \eqref{model1:1}- \eqref{model1:3} admits a  global in time solution $z$ with probability $1.$ Furthermore, $z$ fulfills the following estimate 
\begin{eqnarray}\label{ik21b}
0<e^{\gamma \eta_2 t-A(t)-N_t} \mathcal{T}_t v_0(x)\mathcal{G}(t) \leq z(x,t)\leq  1,\;x\in D, \end{eqnarray}
for any $t\geq 0.$
\end{cor}\label{thm3}
\begin{proof}
By \eqref{ik11}, due to  \eqref{ik22}, we obtain that $\tau_*=\infty.$ Therefore,
 the desired estimate \eqref{ik21b} is valid, since by Remark \eqref{sk1} we  deduce  that $\tau=\infty.$
\end{proof}

In the sequel we derive a sufficient condition for condition \eqref{ik22} to be satisfied. Such a condition is provided in terms of 
 the  principal eigenpair $(\mu_1, \psi_1)$ of the eigenvalue problem \eqref{eig1}--\eqref{eig2} normalized such that \eqref{qp10} holds.
 
We consider initial data $0<z_0(x)=v_0(x)\leq 1,$ such that
\begin{eqnarray}\label{ik25}
0<W_1 \psi_1(x)\leq z_0(x)=v_0(x)\leq 1 , \quad x\in D,
\end{eqnarray}
for some  constant $W_1>0$  to be specified in the sequel.
Remarkably, by virtue of Jentsch's Theorem (see \cite[Theorem V.6.6]{Sc74}), we obtain $\mathcal{T}_{t}\psi_1(x)=e^{-\mu_1 K(t)}\psi_1(x)$ for any $t\geq 0,$.

Set $\psi_m:=\inf_{x\in D} \psi_1>0$, then \eqref{ik25} yields
\begin{eqnarray}\label{ik23}
\mathcal{T}_{t}v_0(x) &&\geq W_1 \mathcal{T}_{t}\left( \psi_1(x)\right)\nonumber\\
&&= W_1 \left(e^{-\mu_1 K(t)} \psi_1(x)\right)\nonumber\\
&&\geq W_1 \psi_m e^{-\mu_1 K(t)},\quad\mbox{for any}\quad x\in D,\; t\geq 0,
\end{eqnarray}
where the lower  bound in \eqref{ik23} is independent of the spatial variable $x.$ 

Since the function $(x,t)\mapsto \mathcal{T}_{t}\phi_1(x)$ is uniformly bounded in $x,$  then \eqref{sk2} thanks to \eqref{ik23} reads
\begin{eqnarray*}
\mu(t)\geq W_1 \psi_m   e^{\gamma \eta_2 t-\mu_1 K(t)-A(t)}\quad\mbox{for any}\quad  t\geq 0,
\end{eqnarray*}
and thus condition \eqref{ik22} is satisfied provided that
\begin{eqnarray*}
\left(W_1 \psi_m \right)^{-3}\int_0^\infty e^{-3\left[\gamma \eta_2 s-\mu_1 K(s)-A(s)- N_s\right]} ds<\frac{1}{4 \lambda \eta_2 \zeta_M},
\end{eqnarray*}
or equivalently
\begin{eqnarray}\label{ik24}
\int_0^\infty e^{-3\left[\gamma \eta_2 s-\mu_1 K(s)-A(s)- N_s\right]} ds<W_2,\qquad
\end{eqnarray}
for $W_2:=\frac{\left(W_1 \psi_m \right)^{3}}{4 \lambda \eta_2 \zeta_M}.$

Therefore we deduce the following global existence result.
\begin{thm}
Under conditions  \eqref{ik25}, \eqref{ik24} for some  $W_1>0,$ Then, problem \eqref{RGLSP1}--\eqref{RGLSP3}, and thus \eqref{model1:1}- \eqref{model1:3} problem as well, has  a global in time solution  with probability $1$ (almost surely).
\end{thm}
\begin{proof}
Note that conditions \eqref{ik25}, \eqref{ik24} imply the vailidity of \eqref{ik22}, and thus the result follows from Corollaary \ref{sk3}.
\end{proof}
\begin{rem}
If $a(t)\equiv 0, b(t)\equiv 1$ and  $k(t)=$constant, for simplicity take $k=\sqrt{2},$ then 
\begin{eqnarray*}
\int_0^\infty e^{-3\left[\gamma \eta_2 s-\mu_1 K(s)-A(s)- N_s\right]} ds=\int_0^\infty e^{3\left((\mu_1-\gamma \eta_2)s+ B^H_s\right)} ds=\infty\quad a.s.,
\end{eqnarray*}
provided that $\gamma<\frac{\mu_1}{\eta_2},$ see \cite{DNMK22}, and thus \eqref{ik24} cannot be satisfied. If 
$\frac{\mu_1}{\eta_2}< \gamma,$ then \eqref{ik24} holds true when the parameter $\lambda$ is sufficiently small; alternatively, if we fix $\lambda$, then  the constant $W_1$  should be chosen  sufficiently large.

On the other hand,  when $a(t)\equiv \sqrt{2\kappa}, b(t)\equiv 0$ and  $k=\sqrt{2},$ then 
\begin{eqnarray*}
\int_0^\infty e^{-3\left[\gamma \eta_2 s-\mu_1 K(s)-A(s)- N_s\right]} ds=\int_0^\infty e^{3\left((\mu_1+\kappa-\gamma \eta_2)s+ B_s\right)} ds=\infty\quad a.s.,
\end{eqnarray*}
by choosing  $\gamma<\frac{\mu_1+\kappa}{\eta_2},$ see \cite{DKN22}, and thus \eqref{ik24} cannot be satisfied. In the complementary case $\gamma>\frac{\mu_1+\kappa}{\eta_2},$ then condition \eqref{ik24} is satisfied for either small $\la$ or for fixed $\la$ and sufficiently large $W_1.$
\end{rem}
\begin{rem}
Notably  the random variables $\tau_*$ and  $\tau^*$,  defined by \eqref{qp9} and  \eqref{ik11} respectively, provide lower and upper bounds for the quenching time $\tau$ of the solution to \eqref{model1:1}- \eqref{model1:3}.
Indeed, consider now intial data  $v_0(x)=W_1 \psi_1(x)$
for some constant $W_1>0.$ Then,  $\tau_*$ and  $\tau^*$  are expressed in terms of the same exponential function of  the stochastic process $N_t.$ Indeed, under that choice for the initial data, we have: 
\begin{eqnarray*}
\mu(t)=W_1 \psi_m   e^{\gamma \eta_2 t-\mu_1 K(t)-A(t)},
\end{eqnarray*}
and
\begin{eqnarray*}
I(0)=\langle v_0 , \psi_1 \rangle_D=W_1  \int_D \psi_1^2(x)\,dx. \end{eqnarray*}
Then, due to \eqref{ik11} and  \eqref{qp9}, we deduce that
\begin{eqnarray*}
\tau_*:=\inf\left\{t>0: \int_0^\infty e^{-3\left[\gamma \eta_2 s-\mu_1 K(s)-A(s)- N_s\right]} ds\geq \frac{W_1^3 \psi_m^3 }{4 \lambda \eta_2 \zeta_M} \right\},
\end{eqnarray*}
and
\begin{eqnarray*}
\tau^*:=\inf\left\{t>0: \int_0^\infty e^{-3\left[\gamma \eta_1 s-\mu_1 K(s)-A(s)- N_s\right]} ds\geq \frac{W_1^3  \left(\int_D \psi_1^2(x)\,dx\right)^3}{3 \lambda \eta_1 \zeta_m}\right\},
\end{eqnarray*}
which implies that $\tau_*\leq \tau^*,$  provided that 
\begin{eqnarray*}
\frac{3 \eta_1\zeta_m}{4 \eta_2 \zeta_M}\psi_m^3= \frac{3 \eta_1\zeta_m}{4 \eta_2 \zeta_M}\left(\inf_{x\in D} \psi_1(x)\right)^3\leq \left(\inf_{x\in D} \psi_1(x)\right)^3  \leq \left(\int_D \psi_1^2(x)\,dx\right)^3.
\end{eqnarray*}
The latter relation is readily seen to be always true since $\eta_1\leq \eta_2$  and $\int_D \psi_1(x)\,dx=1.$
\end{rem}
\section{Applications: MEMS models}\label{cmm}
Deterministic versions of model \eqref{model1:1} -- \eqref{model1:3} with local diffusion have been introduced the last decade to describe the  operation of
certain types of micro-electromechanical systems (MEMS), cf \cite{EGG10, KMS08, KS18, JAP-DHB02}.  MEMS devices are precision devices which integrate mechanical processes
with electrical circuits. Their size ranges from
millimetres down to microns, and involve precision mechanical
components which  can be constructed using semiconductor
manufacturing technologies \cite{KS18, JAP-DHB02,Younis}. MEMS devices are commonly employed in biomedical engineering applications, including the design of micro-scale drug delivery devices and the development of micropumps for microfluidic diagnostic tools, among others \cite{Chirkov, Nuxoll, Yager}.

The key part of such an electrostatically actuated
MEMS device usually consists of  an
elastic plate (or membrane) suspended above a rigid ground one.
Regularly the elastic plate is held
fixed at two ends while the other two edges remain free to move  (see Figure~ \ref{Figmems1}).
 \begin{figure}[h!]
  \begin{center}
\includegraphics[width=.7\textwidth]{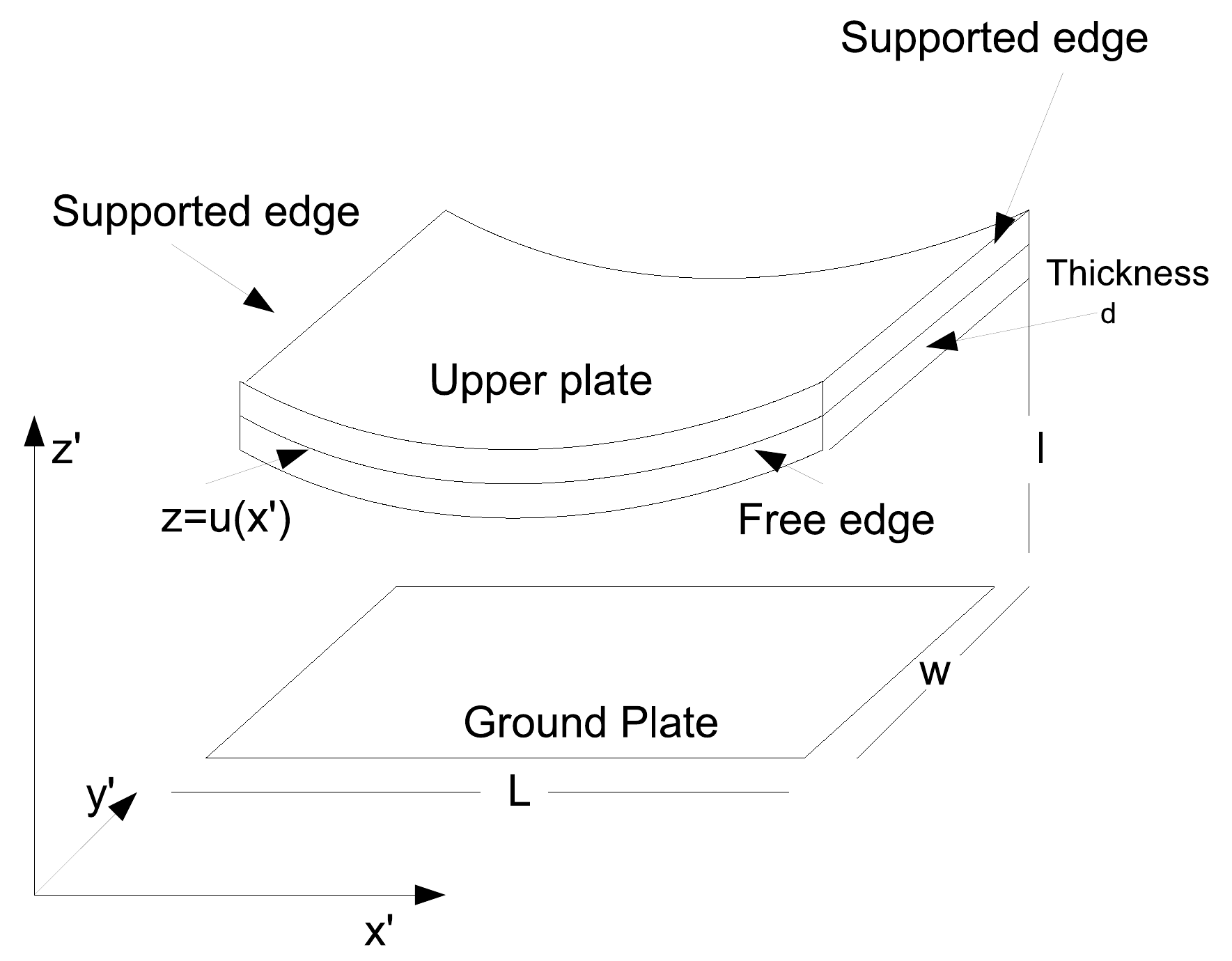}\medskip
\end{center}
  \caption{\it Schematic representation of a MEMS device }
\label{Figmems1}
\end{figure}
A potential difference
$V$  is applied  between the elastic membrane and the rigid ground plate,
leading to a deflection of the  membrane towards the plate. Considering   the width $d$ of the intermediate gap, i.e the gap between the membrane
and the bottom plate, to be small compared to the
device length $L$,  then
the deformation of the elastic membrane $u,$ after proper scaling,  is given by the dimensionless equation
\begin{eqnarray}\label{ik3a}
\frac{\partial u}{\partial t}=\Delta u +\frac{\lambda\, q(x)}{(1-u)^2},
\quad x\in D,\; t>0, 
\end{eqnarray}
see \cite{KS18, JAP-DHB02, PT01}, where the laplacian term $\Delta u$ describes to the spread of the deformation, whilst the term $(1-u)^{-2}$ arises as a consequence of the electrostatic features of the MEMS device. Here, the term  $q(x)$ describes the varying dielectric properties of the membrane and for some elastic materials can be taken to be constant;
for simplicity, henceforth  we  assume  that $q(x)\equiv 1.$ 
Furthermore, the parameter $\la$  appearing in \eqref{ik3a} equals to
$
\la=\frac{V^2 L^2 \varepsilon_0}{2\widetilde{T}\ell^3},
$
and  is  the {\it tuning} parameter of  the  device.
Note that $\widetilde{T}$ stands for the tension of the elastic membrane,
 $\ell$ is the characteristic width of the gap
between the membrane and the fixed ground plate (electrode),
whilst  $\varepsilon_0$ is the permittivity of free space. For extending further the stable operation of the MEMS device then a capacitance connected in series with MEMS is  introduced to the underlying electrical circuit.  Then we are led, via the application of Kirchoff's laws, into  versions of model \eqref{ik3a} involving nonlocal reaction terms, cf. \cite{KS18, JAP-DHB02, PT01}.

MEMS engineers are commonly  interested in identifying the conditions under which the elastic membrane can touch the rigid plate, a mechanical phenomenon  usually referred to as {\it touching down} and one that can potentially lead to the destruction of the device. Touching down can be described via model \eqref{ik3a}, and its nonlocal variations. It actually corresponds to the case when the deformation $u$  reaches the value $1;$ such a situation is known as {\it quenching (or extinction)} in the mathematical literature.

 Experimental observations (see \cite{Younis})  show a significant uncertainty regarding the values of $V$ and $\widetilde{T}.$  In that case, incorporating this uncertainty into the tunning parameter $\la$ we can obtain some a stochastic model  with a multiplicative white noise involving  \cite{DKN22, Kavallaris2016}.  In case of a MEMS device with a long-range dependence uncertainty for $V$ and $\widetilde{T}$ it seems reasonable to consider a stochastic model with multiplicative fractional noise with Hurst index $1/2<H<1,$ cf. \cite{DNMK22}.  

A more complex configuration of the MEMS device described in Figure\ref{Figmems1}  incorporates the case where two edges of the membrane are attached to a pair of torsional and translational springs, modeling a flexible, non-ideal support, see \cite{DKN19, Younis}. In that case homogeneous Robin boundary conditions should be assigned to equation \eqref{ik3a}, cf.  \cite{DKN19, DNMK22}.

Now consider the case we would also like to model nonlocal effects of the elastic membrane,  maybe due to material discontinuities. In that case the spread of the membrane's deflection should be modeled, into equation \eqref{ik3a}, by a nonlocal diffusion operator $\mathcal{L}$ of the form \eqref{nop}, with a possibly time-dependent diffusion coefficient, incorporating the natural treatment of balance laws on and off material discontinuities, cf. \cite{DDGL17, Du19}.
To integrate any possible material anomalies on the boundary a nonlocal Robin boundary condition of the form \eqref{model1:2} should be assigned to the underlying nonlocal equation. A rather simple choice is to take  $\mathcal{L}=(-\Delta)^{\alpha}, \, 0<\alpha<1,$ whilst the nonlocal operator $\mathcal{N}$ given by \eqref{bnop}, should have kernel of the form $|x-y|^{-d-2 \alpha}.$ Finally, in order to include elastic membranes could  generate long-range and short-range  voltage and tension fluctuations (occuring as clustering)  we could embody a noise of the form \eqref{cnt}.

For such a model with extended Robin conditions the analysis and the analytical results and estimates  of sections \ref{lex}-\ref{gexlb} give an insight about the quenching behaviour of the model. For the case of the extended Dirichlet conditions, where no similar  analytical results are available,  we present a preliminary numerical treatment of the problen in the following section \ref{nss}.

\section{Numerical solution}\label{nss}
 In the current section we deliver a numerical study of the following problem 
  \begin{eqnarray}
&&du =\left(-(-\Delta)^{\alpha} u +\frac{\lambda}{\left(1-u\right)^2 }\right)dt+ \pi(u) dN_t,
 \quad x\in D,\; t>0, \; 0<\alpha<1, \quad \label{NLSP1}\\
&& u(x,t)=0, \;  x\in D^c,\; t>0,\label{NLSP2}\\
&& 0\leq u(x,0)=u_0(x)<1, \quad x \in D,\label{NLSP3}
\end{eqnarray} 
for  the one-dimensional   case $d=1.$  The considered  time-dependent noise is a combination of standard and  fractional Brownian motions. This case corresponds to a simpler version of the more general problem studied in the previous sections and it is closely related to the MEMS application described in the preceding section.
In relation to problem (\ref{model1:1})-(\ref{model1:3}), here we consider,  $\mathcal{L}=(-\Delta)^{\alpha}, \, 0<\alpha<1,$, $u=1-z,$ $D=[-1,1],$ $k(t)=2$, $\gamma=0$, $\zeta(x)=1$ and extended Dirichlet boundary conditions. 
  
The  considered noise term is of a multiplicative  form, that is  $\pi(u)\,d N_t$ for
 $\pi(u)= (1-u)$. In particular we take the noise being of the form  $\pi(u)\,d N_t=\kappa_1(1-u)\,d B_t+\kappa_2(1-u)\,d B^H_t $ for some positive constants $\kappa_1, \kappa_2;$ that is a mixture of the standard and fractional Brownian motion.
  
Here we focus in  the case of homogeneous Dirichlet  conditions in the complement of the interval  $[-1,\,1].$ 
 Note that homogeneous Dirichlet boundary condition $u=0$  in $[-1,\,1]^c$ corresponds in having $z=1$ for  $x\in [-1,\,1]^c ,$ and  this case is not  actually  covered by the  analysis in the previous sections.

 The consideration of extended Dirichlet boundary conditions is important, since those conditions 
 are quite relevant to the MEMS application 
 (see for example ~\cite{EGG10, KS18, JAP-DHB02}) considered in the previous section. On the other hand,   since as for the case of the standard Laplacian operator the analytical methods for estimating the quenching probability are more delicate (see \cite{DKN22, DNMK22}) an initial numerical study of the Dirichlet problem is a valuable contribution.  So we attempt a numerical study of the conisdered stochastic nonlocal model in that case. We intend to derive an initial estimation of the dynamics of our nonlocal stochastic problem under the infuence of extended Dirichlet boundary conditions. 

\subsection{Finite Differences  approximation}
In order to approximate numerically the problem \eqref{NLSP1}-\eqref{NLSP3}
   we proceed with a finite difference semi--implicit Euler in time  scheme, cf. \cite{Duo18, Lord}.

We set a discretization in $[0,T]\times [-1,1]$, $0\leq t\leq T$, $0\leq x\leq 1$ with
$t_n=n\delta t$, $\delta t=\left[{T}/{N}\right]$  for $N$ the number of time steps and we also introduce the grid points in $[-1,1]$, $x_j = -1+j\delta x$, for $\delta x = 2/M$ and $j = 0,1,\ldots, M$.

 Then, initially we apply  a finite differences approximation for the   fractional Laplacian term $(-\Delta)^{\alpha} u$ in  (\ref{NLSP1}). We follow the finite differences approximation approach  given in  \cite{Duo18}.
 
Following this approach we obtain 
$(-\Delta)^{\alpha}_{\rho} u \simeq Au$, for $A$ an $(M-1)\times (M-1)$ matrix having the form: 

\begin{eqnarray}
A_{i,j}=C_{1,\alpha}\left\{\begin{array}{ccc}
\sum_{k=2}^{M-1}\frac{(k+1)^\chi-(k-1)^\chi}{k^\rho} 
+(\frac{(M+1)^\chi-(M-1)^\chi}{M^{\rho}})
+\left(2^\chi +\kappa_\rho -1\right) + \frac{\chi}{\alpha M^{2\alpha}}, &&  i=j,\\  {}&&{} \\
-\frac{\left(|j-1|+1\right)^\chi-\left(|j-1|-1\right)^\chi}{2 |j-1|^\rho}, &&  j \neq i,\, i\pm 1,\,\,\\ {}&& {}\\
-\frac12\left(2^\chi +\kappa_\rho -1\right),  &&  j = i\pm 1,
\end{array} \right.
\end{eqnarray}
where here we set  $\chi=\rho-2\alpha$,  $\kappa_\rho=1+2 \alpha$ for $2\alpha\in (1,2)$ and $\kappa_\rho=1$ for $2\alpha=1$. Also the parameter $\rho\in (2\alpha, 2]$ and in our simulations we take  $\rho=1+\alpha$. For more details about the derivation of the scheme and the choices of the various parameters we refer to \cite{Duo18}.

We then apply a semi-implicit Euler method in time. We denote by
  $U_h^n=(u_0^n,u_1^n,\ldots, u_M^n)$ the finite difference approximation of $u(x,t)$, i.e. $u(x,t^n)\simeq U_h^n$, and by $u_h^n=(u_1^n,\ldots, u_{M-1}^n)$ since due to Dirichlet boundary conditions $u_0^n=u_M^n=0$.
  
By a standard discretization in time we obtain 
\bgee  
d u(t_{n})\simeq \frac{\left(u_h^{n+1}- u_h^{n}\right)}{\delta t}= A u_h^{n+1} +  g\left(u_h^n\right)+\pi\left(u_h^n\right)\left(\kappa_1 b_s(t_n)+\kappa_2 b_s^H(t_n)\right),
\egee
   or
\bgee
\left( I-\delta t A\right)u_h^{n+1} =u_h^{n}+ \delta t\, g\left(u_h^n)\right)
   +\delta t \, B_h(t_n),
\egee
  where we denote  by $B_h$  $\in \mathbb{R}^{M-1}$ the vector 
  $B=\pi\left(u_h^n\right)\left(\kappa_1 b_s(t_n)+\kappa_2 b_s^H(t_n)\right)$.
  Also $b_s(t)=( \xi(t_{n+1})-\xi(t_{n}))/\sqrt{\delta t}$ where $b_s(t_n)\sim N(0,1)$ are i.i.d. random variables for i.i.d.  standard Brownian motions $\xi(t)$.
  Similarly we sample the fractional Brownian motion  by considering 
   i.i.d.  fractional Brownian motions $\xi^H(t)$. The latter are sampled by the circulant embedding method with a standard routine (e.g. see \cite{Lord}, Chapter 6). 
   
 \subsection{Simulations}

 Initially, we present a realization of the numerical solution of problem \eqref{NLSP1}-\eqref{NLSP3} in Figure \ref{Fig_sim1}(a) for $\lambda=0.1$, $\kappa_1,\kappa_2=0.1$,     initial condition $u(x,0)=c\,(1-x^2)$ for $c=0.1$ and extended homogeneous Dirichlet boundary conditions. 
 We also consider the nonlocal diffusion exponent $\alpha=0.6$ and take Hurst index $H=0.6$. By this performed realization the occurrence of  quenching is evident. 

 For a different set of five realizations but for the same parameters in Figure \ref{Fig_sim1}(b) the maximum of the solution at each time step is plotted and again a similar quenching behaviour is observed.
\begin{figure}[!htb]\vspace{-4cm}\hspace{-1cm}
   \begin{minipage}{0.48\textwidth}
     \centering
     \includegraphics[width=1.2\linewidth]{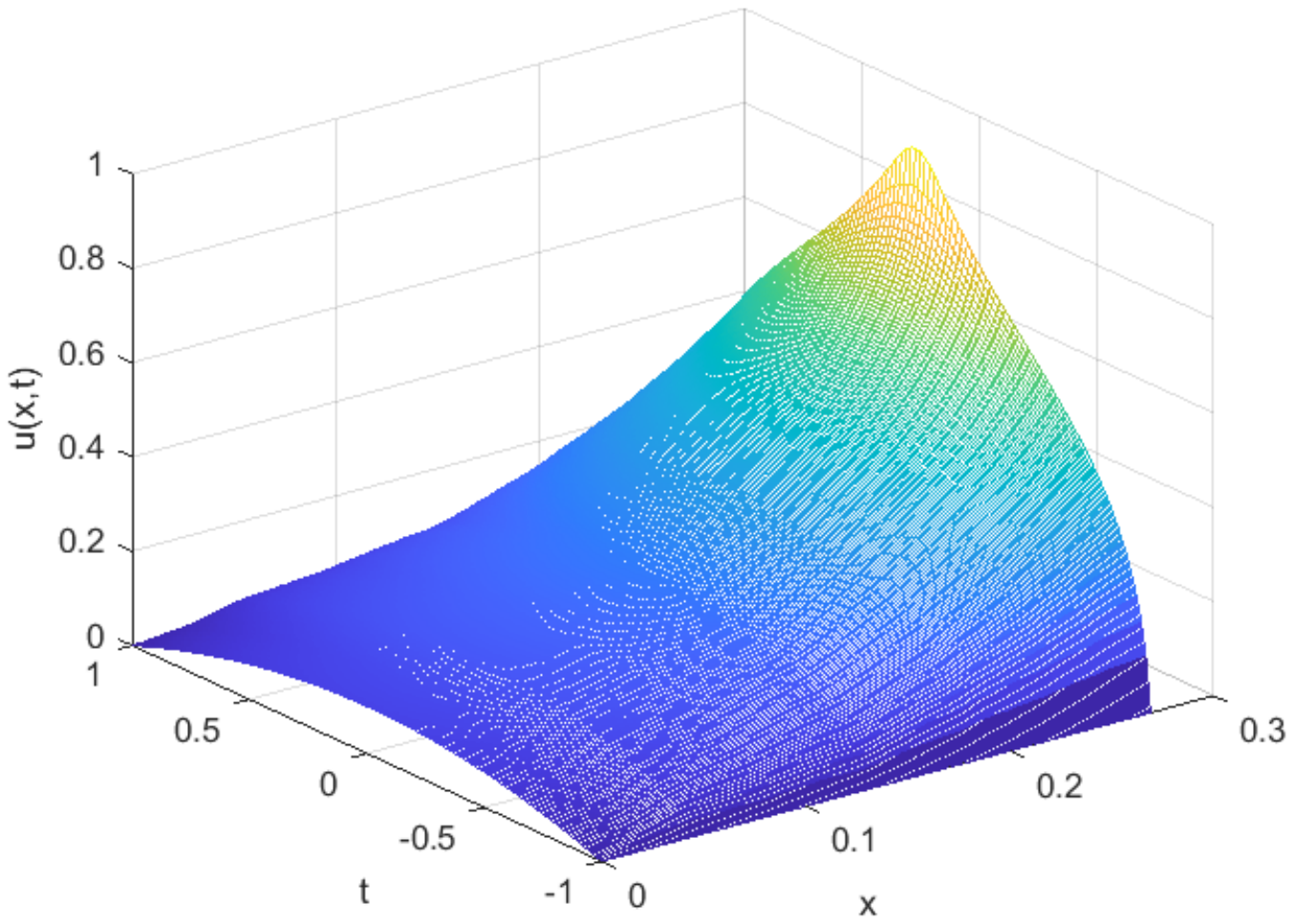}
   \end{minipage}\hfill
   \begin{minipage}{0.48\textwidth}\hspace{-2cm}
     \centering
     \includegraphics[width=1.2\linewidth]{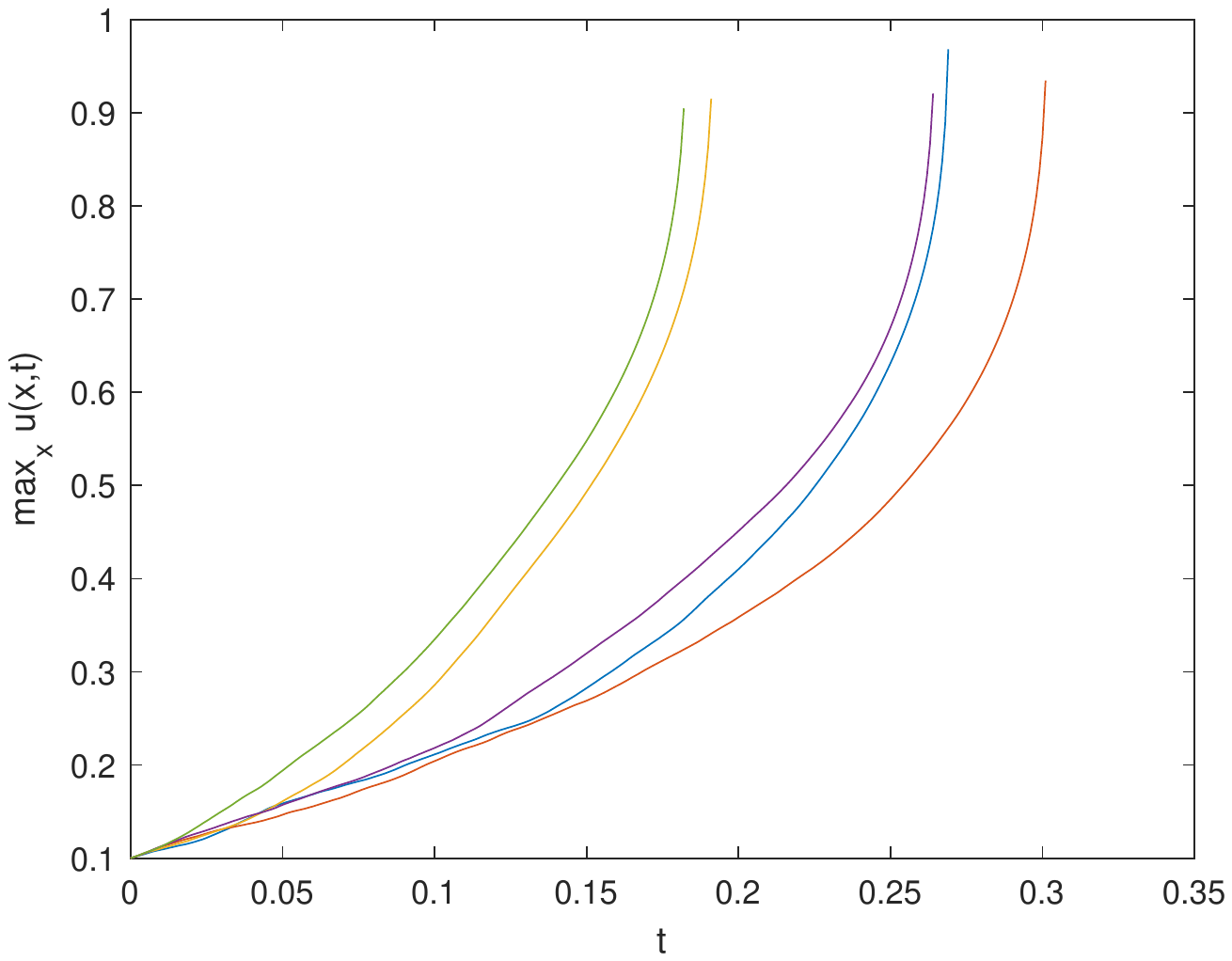}
   \end{minipage}\vspace{-4cm}
   \caption{(a) Realisation of the numerical solution of problem \eqref{NLSP1}-\eqref{NLSP3} for $\lambda=0.1$, $\kappa_1=\kappa_2=0.1$, $M=101$,  $N=10^4$,    and initial condition $u(x,0)=c\,(1-x^2)$ for $c=0.1$.
   (b) Plot of $\|u(\cdot,t) \|_\infty$ from a different set of five realizations but with the same    parameter values.}\label{Fig_sim1}
\end{figure}
  An interesting aspect worth investigating is
 the derivation of  estimates  of the quenching probability  in a specific
time interval $[0, T]$ for some $T>0.$

We know, see \cite{KMS08, KS18}, that  for the coresponding deterministic  problem with the  Laplacian operator, that is  \eqref{NLSP1}-\eqref{NLSP3} with $\alpha=1$ and $\pi(u)\equiv 0$ and imposed Dirichlet boundary conditions, then the  solution $u$ will eventually quench in some finite time $T_q$ for large enough values of the parameter $\lambda$ or big enough initial data.

From the application  point of view, an estimate of the probability that $T_q<T$ with respect to various values of the parameter $\lambda$  would be desirable. In Table $(T1)$ the results of such a numerical experiment are presented. In particular, implementing  $N_R$ realizations,  in the first column we print out  the values of the parameter $\lambda$  considered.  The second column contains the number of times that the solution quenched before the specified simulation time $T$ over the number of realizations which provides the estimation of the quencing probability.  Additionally,  in the last two columns the mean $m(T_q)$ and the variance $Var (T_q)$ of the quenching time respectively are given.
 The quenching time $T_q$ numerically is approximated as $T_q\simeq t_m$ for $t_m$ the  maximum time step for which the condition $\max_j\left(u(x_j,t_m)\right)\leq 1-\varepsilon$ holds
     for a predefined small number  $\varepsilon$. In  the  simulations of this section $\varepsilon$  is taken to be the machine tolerance i.e. $\varepsilon=2.2204\, 10^{-16}$. 
The rest of the parameters were taken to be the same as in the previous simulations but with 
$H=0.7$, $\alpha=0.6$, $\kappa_1=\kappa_2=0.1$, $M=41$.

\begin{center}\label{Table1}
{\bf Table (T1)}\\Realizations of the numerical solution of problem \eqref{NLSP1}-\eqref{NLSP3} \\
for $N_R=10000$ in the time interval $[-1,\,1].$ \\
\bigskip
 \begin{tabular}{||c||c|c|c||}\hline\label{T1}
    $\lambda $      &        Quenching  Probability      & $m(T_q)$  &  $\sigma^2 (T_q) $\\
      \hline
        0.01   &  0 & - & -\\   
  0.2   &  0.1802 & 0.8285 & 0.0124\\   
 0.4 &    0.5141 & 0.6953 & 0.0200 \\   
0.6 &  0.8021 & 0.5482  & 0.0182  \\   
0.8  &  0.9542 &  0.4145 & 0.0091  \\
1   & 0.9953 &0.3188  & 0.0026 \\
1.2   & 0.9997 &0.2583  &  6.5095e-04 \\
1.4   &1.0000 & 0.2192  & 2.2777e-04\\
 \hline
\end{tabular}
\end{center}
 \bigskip
 
From the results presented in Table $(T1)$ we deduce that we have a behaviour of the problem resembling the deterministic 
case. Increasing $\lambda$ results in a corresponding increase of the quenching events while the estimated quenching time decreases. Finally for large enough value of the parameter $\lambda$ we have quenching with estimated probability one while for very small $\lambda=0.01$ we have no quenching events in the interval $[0,T]$.

Next we proceed with an investigation of the effect of the  regularizing term $\gamma(1-u)$ on the quenching behaviour.  We solve numerically the problem (\eqref{NLSP1}-\eqref{NLSP3}) using the same set of parameters as for the experiments in Table $(T1)$ but with the source term having now the form  $\frac{\lambda}{\left(1-u\right)^2 } -\gamma(1-u)$, with $\gamma=0.1$. The results are demonstrated in Table $(T2).$

 \bigskip
  \begin{center}\label{Table2}
{\bf Table (T2)}\\Realizations of the numerical solution of problem \eqref{NLSP1}-\eqref{NLSP3} with the addition of the regularizing term $\gamma(1-u)$. \\
for $N_R=10000$ and $\gamma=0.1$, in the time interval $[-1,1].$ \\
\bigskip
 \begin{tabular}{||c||c|c|c||}\hline\label{T2}
    $\lambda $      &        Quenching  Probability       & $m(T_q)$  &  $\sigma^2 (T_q) $\\
      \hline
       0.01   &  0 & - & -\\   
  0.2   &0.1470   & 0.8388  &  0.0112 \\   
 0.4 &  0.4614   &0.7154  & 0.0198  \\  
0.6 & 0.7473   &0.5729   &0.0192   \\    
0.8  & 0.9274  & 0.4386   &0.0113   \\ 
1   & 0.9907  &0.3350   &  0.0038  \\ 
1.2   &0.9998  &0.2683   & 8.5834e-04 \\ 
1.4   &1.0000  &0.2259   & 2.9059e-04 \\ 
 \hline
\end{tabular}
\end{center}
\bigskip

Indeed comparing these results with those of  Table $(T1)$ we observe that the probability of quenching decreases due to the addition of the regularizing term. Also the mean value of the quenching time is larger with slightly smaller variation.
This is a result combatible with those obtained analyticaly for the extended Robin conditions case.


In the next set of experiments in Table $(T3)$, and in the rest of this section, we drop the regularizing term i.e. we set $\gamma=0$ and  focus on the effect of the fractional Brownian noise on the quenching behaviour of the problem. Namely for a fixed value of $\lambda=0.4$ and of $\kappa_1=0.1$ we vary the coefficient $\kappa_2$ expressing the intensity of the  fractional Brownian term. The rest of the parameters are kept the same as in the previous simulations.
 \bigskip
  \begin{center}\label{Table3}
{\bf  Table (T3)}\\Realizations of the numerical solution of problem \eqref{NLSP1}-\eqref{NLSP3} in the case of varying fractional Brownian noise intensity  for $N_R=10000$ in the time interval $[-1,\,1].$ \bigskip \\
 \begin{tabular}{||c||c|c|c||}
 \hline\label{T3}
    $\kappa_2 $      &        Quenching   Probability      & $m(T_q)$  &  $\sigma^2 (T_q) $\\
  \hline
     0.05  & 0.5191 &0.7179 & 0.0177 \\
0.1 &  0.5205 & 0.6961 & 0.0196\\
0.5 &     0.5341    	&   0.5244 &  0.0285 \\  
1 &      0.5575 	&    0.4453  & 0.0321\\ 
1.5 &      0.5702 	&    0.3989  &  0.0344\\ 
2 &      0.5799	&   0.3718  & 0.0354\\ 
 \hline
\end{tabular}
\end{center}
 \bigskip

We observe that by increasing the coefficient $\kappa_2$ we have a tendency to obtain more quenching events  and consequently a decreased average quenching time $T_q$. Such a result  should be expected due to the long-range dependence is exhibited by the fractional Brownian motion for $H=0.7>1/2.$ 

Moreover in the next set of graphs, see Figure \ref{Fig_sim2},  we investigate the effect of  the Hurst  index  $H$ and the order $\alpha$ of the fractional Laplacian operator upon the  quenching probability and the quenching time. We set $H\in [1/2,\, 1]$ and $\alpha \in [0.1,\,0.9]$ and a partition in these intervals with step $\delta_s=0.05$, while for the rest of the parameters we have $\lambda=0.4$, $\kappa_1=\kappa_2=0.1$ and  $M=41$. The number of realisations for each choise of $(\alpha_\ell, H_k)$ where $\alpha_\ell=0.5+\delta_s \ell$, $\ell=0\ldots 10 $, $H_k=0.1+\delta_s k$, $k=0\ldots 16 $, and again  $N_R=10000$.
\begin{figure}[!htb]\vspace{-4cm}\hspace{-1cm}
   \begin{minipage}{0.48\textwidth}
     \centering
     \includegraphics[width=1.2\linewidth]{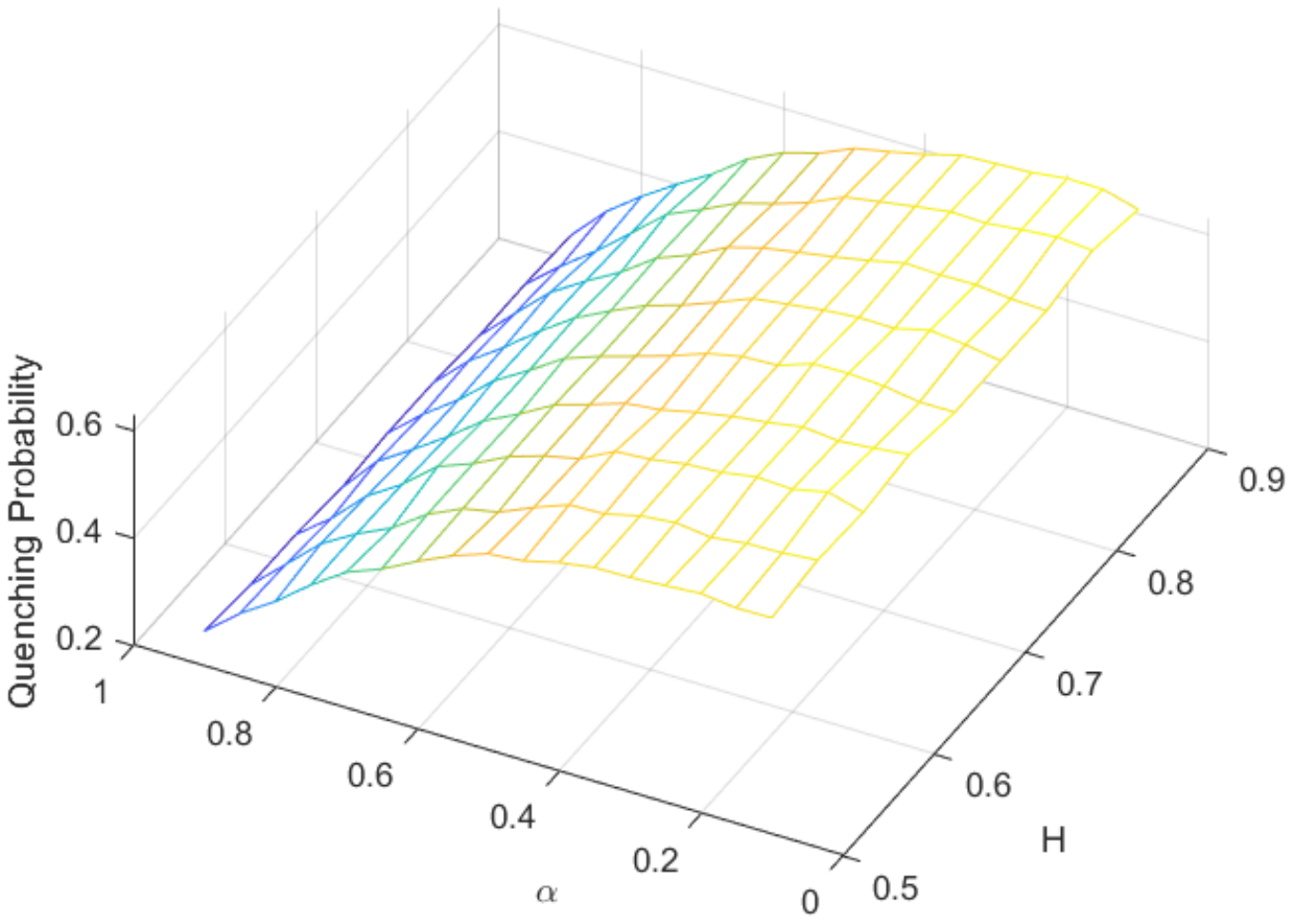}
   \end{minipage}\hfill
   \begin{minipage}{0.48\textwidth}\hspace{-2cm}
     \centering
     \includegraphics[width=1.2\linewidth]{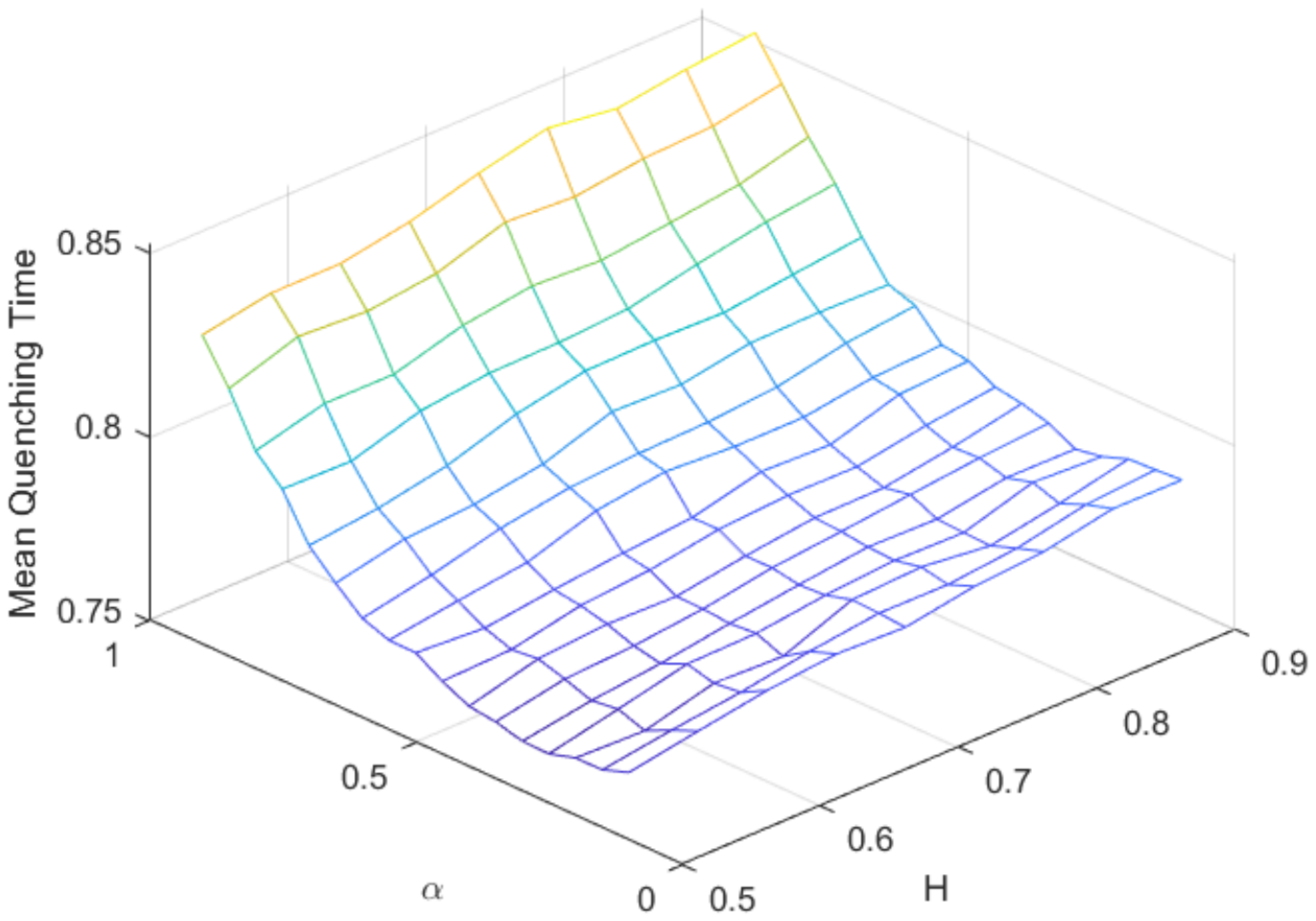}
   \end{minipage}\vspace{-4cm}
   \caption{(a) Graph of the quenching probability in the time interval $[-1,\, 1]$ of problem \eqref{NLSP1}-\eqref{NLSP3} with respect to the parameters $\alpha$ and $H$,  for $\lambda=0.4$, $\kappa_1=\kappa_2=0.5$, $M=41$,  $N=10^4$,    and initial condition $u(x,0)=c\,(1-x^2)$ for $c=0.1$. 
   (b) Plot of mean value of the quenching time $T_q$ with respect  to the parameters $\alpha$ and $H$ for the same set of experiments.}\label{Fig_sim2}
\end{figure}

 Figure \ref{Fig_sim2}(a) shows that the quenching probability decreases as the Hurst index $H$ increases from $\frac12$ to $1,$ while there is no significant variation of the probability with respect to the order of the fractional Laplacian $\alpha.$ In Figure \ref{Fig_sim2}(b) we observe that the mean quenching time increases as both the Hurst index $H$ and the order of the fractional Laplacian $\alpha$ increase.

\section{Discussion}
 In the current work, we investigated the behaviour of certain stochastic partial differential equations, including nonlocal diffusion operators assigned to nonlocal Robin conditions.  The dynamics were driven by a mixture of a Wiener process and fractional Brownian motion with Hurst index $H>1/2.$ The considered model can describe the operation of MEMS devices embeded to materials with  discontinuities. This model can exibit the phenomenon of {\it finite time quenching}, which is closely related to the mechanical phenomenon of {\it touching down} in MEMS devices. Similar models, related also to MEMS systems, driven by either  single Brownian or  fractional Brownian motions have been investigated in \cite{DKN22, DNMK22, Kavallaris2016}. However, to the best of our knowledge, this is the first time in the literature that the quenching behaviour of such nonlocal  model with mixed noises is investigated.

Our theoretical analysis enabled us to provide upper bounds for the {\it quenching probability} and for the {\it quenching time}. To this end, we employed  estimates of perpetual integral functionals of Brownian motion as wells as  recently obtained tale estimates  of fractional Brownian motion. Moreover, 
we were able to estimate the probability of global existence, and as a by-product we also derived lower estimates of the quenching time.
The produced analytical results exhibit the strong impact of the noise to the long time dynamics of the model. In particular, as alluded to in Theorem \ref{cvn3}, the form of the nonlinear term  in problem \eqref{model1:1} -- \eqref{model1:3} forces the solution towards quenching almost surely. This is in stark contrast with the dynamics of the corresponding deterministic local diffusion system, cf.  \cite{KMS08,  KS18}. A key auxiliary result for our analytical approach is the positivity of the principal Robin eigenpair of the considered nonlocal operator. It is worth emphasizing, that since such a result was not available in the literature we had to prove it, see Appendix.

A complementary numerical study of the underlying nonlocal model has been employed for the case of fractional Laplacian $\left(-\Delta\right)^{\alpha}, \; 0<\alpha<1,$ and for homogeneous extended  Dirichlet boundary conditions. The presented numerical experiments  shed light to the impact of the mixed noise  to the dynamics of the model. More interestingly, the numerical simulations illuminate how the dispersal coefficient $0<\alpha<1$ affects the quenching behaviour of the underlying model, 
an effect cannot be observed by means of an analytical approach. 
\section{Appendix}\label{apx}



\begin{proof}[Proof of Proposition \ref{peip}]
The proof (that follows the strategy of \cite{Ev}) can be broken into the following steps:

1. We first now that  by the Green-Gauss formula, for any $u,v$ it follows that
\begin{eqnarray*}
Q_{\beta}(u,v)&&={\mathcal E}(u,v) + \int_{D^{c}} \beta u(y)v(y) dy \\
&&= \int_{D} {\mathcal L} u(x) v(x) dx + \int_{D^{c}} {\mathcal N} u(y) v(y) dy + \int_{D^{c}} \beta(y) u(y)v(y) dy  \\
&&=\int_{D} {\mathcal L} u(x) v(x) dx,
\end{eqnarray*}
as long as $u$ satisfies the nonlocal Robin-type boundary conditions.

Hence, for any $u$ satisfying Robin-type boundary conditions
\begin{eqnarray}\label{C}
Q_{\beta}(u,v)= \int_{D} {\mathcal L}u(x) v(x) dx = ( {\mathcal L}u, v)_{L^{2}(D)}.
\end{eqnarray}

2. We now take any $u \in L^{2}(D)$, such that $\| u \|_{L^{2}(D)}=1$. Since the set of Robin eigenfunctions for ${\mathcal L}$ forms an orthonormal basis (see Theorem 4.21 in \cite{foghem2022general}) we expect an expansion for $u$ of the form $u= \sum_{k \in {\N}} d_{k} \psi_{k}$, with $d_{k}=(u, \psi_{k})_{L^{2}(D)}$, and the expansion converging in $L^{2}(D)$ and $\sum_{k \in \N} d_{k}^2 =1$.

Moreover, for any two eigenfunctions $\psi_{k}$, $\psi_{n}$, we have (using \eqref{C}) that
\begin{eqnarray}\label{D}
Q_{\beta}(\psi_{k},\psi_{n}) = (\mathcal{L}\psi_{k}, \psi_{n})_{L^2(D)} = \mu_{k} (\psi_{k}, \psi_{n})_{L^2(D)} = \mu_{k} \delta_{k,n},
\end{eqnarray}
where we also used the orthonormality of $\{ \psi_n\}$ in $L^{2}(D)$.

Using arguments related to the Lax-Milgram lemma, we can consider the bilinear form $Q_{\beta}$ as forming an inner product $( u, v)_{\beta}:=Q_{\beta}(u,v)$ on the fractional Sobolev space  $V_{\nu}(D \mid \R^d)$, which will subsequently be denoted by $H_{\nu}^{1}(D)$ to indicate the resemblance with the standard Sobolev space used for the local problem.
  To this end, motivated by \eqref{D}, we note that the set $\{ \psi_k \}:=\{ \psi_k/ \mu_{k}^{1/2}\}$ forms a basis for $H^{1}_{\nu}(D).$  Indeed, \eqref{D} shows that $\{\psi_k\}$ is orthonormal with respect to this new inner product.
To show that this is a basis for $H^{1}_{\nu}(D)$ it suffices to show that  $(\psi_{k}, v )_{\beta}=0$ for all $k$ implies that $v=0$.   To see this we simply have to use \eqref{C}, to observe that
\begin{eqnarray*}
0=(\psi_{k}, v)_{\beta} = Q_{\beta}(\psi_{k}/\mu_{k}^{1/2}, v) = \mu_{k}^{1/2} ( \psi_{k}, v )_{L^2(D)}, \,\,\, \forall \,\, k \in \N,
\end{eqnarray*} 
and by the fact that $\{ \psi_k \}$ is an orthonormal basis for $L^2(D)$, we see that $v=0$ as required. 

Hence, $\{ \psi_{k} \} = \{ \psi_k/ \mu_{k}^{1/2}\}$ is an orthonormal  basis for the Hilbert space $H^1_{\nu}$, endowed with the inner product $( \cdot,  \cdot )_{\beta} = Q_{\beta}(\cdot, \cdot)$. This implies that for any $u \in H^1_{\nu}(D)$,  the series $u= \sum_{k \in N} c_{k} \psi_{k} = \sum_{k \in \N} c_{k} \psi_{k}/\mu_{k}^{1/2}$, for $c_{k}=(u, \psi_{k})_{\beta} = Q_{\beta}(u, \psi_{k}/\mu_{k}^{1/2})$ converges in $H^{1}_{\nu}(D)$. We easily see that $ Q_{\beta}(u, \psi_{k}/\mu_{k}^{1/2}) = (u,\psi_{k})_{L^{2}(D)}$. By the observation the $ H^{1}_{\nu}(D) \hookrightarrow L^2(D)$, so that $u$ also admits an expansion $u = \sum_{k \in \N} d_{k} \psi_{k}$, $d_{k}=(u, \psi_{k})_{L^2(D)}$ in $L^{2}(D)$, and upon direct comparison with the corresponding expansion for $u$ in $H^{1}_{\nu}$, we see that the two expansions coincide.

3. We are now ready to show the variational  formula \eqref{Va}.

Consider any $u \in H^{1}_{\nu}(D) \hookrightarrow L^2(D)$, such that $\| u \|_{L^2(D)}=1$. By the results of step 2, it admits an expansion of the form $u= \sum_{k \in \N} d_{k} \psi_{k}$, with $d_{k}=(u, \psi_{k})_{L^2(D)}$ and $\sum_{k \in \N} d_{k}^2 =1$.  We now calculate $Q_{\beta}(u,u)$ using this expansion and the bilinearity of $Q_{\beta}$:
\begin{eqnarray*}
Q_{\beta}(u,u)=Q_{\beta}\left(\sum_{k \in \N} d_{k} \psi_{k}, \sum_{k' \in \N} d_{k'} \psi_{k'}\right) = \sum_{(k,k') \in \N \times \N} d_{k} d_{k'} Q_{\beta}(\psi_{k},\psi_{k'}) \\
\stackrel{(\ref{D})}{=}   \sum_{(k,k') \in \N \times \N} \mu_{k} d_{k} d_{k'}  \delta_{k,k'} = \sum_{k \in \N} d_{k}^2 \mu_{k} \ge \mu_1 ( \sum_{k \in \N} d_{k}^2 ) = \mu_1,
\end{eqnarray*}
where we used the fact that $\sum_{k \in \N} d_k^2=1$ and $0 < \mu_1 \le \mu_2 \le \cdots $. 

Hence, for any $u \in H^{1}_{\nu}(D)$ such that $\| u\|_{L^2(D)}=1$ it holds that $Q_{\beta}(u,u) \ge \mu_1$. Choosing, $u=\psi_1$, the first Robin eigenfunction we see that $Q_{\beta}(u,u)=Q_{\beta}(\psi_1,\psi_1)=\mu_1$. Hence, \eqref{Va} is proven.

4. We now show that  for any $u \in H^1_{\nu}(D)$ such that $\|u \|_{L^2(D)}=1$ the following are equivalent:
\bge
&&{\mathcal L} u = \mu_1 u\label{E1} \\
&&{\mathcal N} u + \beta u =0\label{E1A}
\ege
and
\begin{equation}\label{E2}
Q_{\beta}(u,u)=\mu_1
\end{equation}

Clearly \eqref{E1}--\eqref{E1A} implies \eqref{E2} by a simple application of the Green-Gauss formula. It thus remains to prove the reverse implication. 

Suppose that \eqref{E2} holds for some $u \in H^1_{\nu}(D)$ such that $\|u \|_{L^2(D)}=1$. By step 2, $u$ admits an expansion of the form $u =\sum_{k \in \N} d_{k} \psi_{k}$ with $d_{k}=( u, \psi_{k})$, and $\sum_{k \in \N} d_{k}^2=1$.   Moreover, by  the bilinearity of $Q_{\beta}$, and the above expansion we have that (essentially as in step 3) that
\begin{eqnarray}\label{F}
Q_{\beta}(u,u)=\sum_{k\in \N} d_{k}^2 \mu_{k}.
\end{eqnarray}

Note that
\begin{equation}
\begin{aligned}
\sum_{k \in \N} d_{k}^2 \mu_1 = \left( \sum_{k\in \N} d_{k}^2 \right) \mu_1 \stackrel{ \| u\|_{L^2(D)}=1}{=} \mu_{1} \stackrel{\eqref{E2}}{=} Q_{\beta}(u,u)  \stackrel{\eqref{F}}{=} \sum_{k\in \N} d_{k}^2 \mu_{k},
\end{aligned}
\end{equation}
which upon rearrangement yields,
\begin{eqnarray}
\sum_{k \in \N} d_{k}^2 (\mu_k - \mu_1) =0,
\end{eqnarray}
and in return (recalling the fact that $0 < \mu_1 \le \mu_2 \le \cdots $) yields
\begin{eqnarray}
d_{k}=(u, \psi_{k})_{L^2(D)} =0, \,\,\, \mu_k > \mu_1.
\end{eqnarray}
Combining that with the expansion for $u$ we see that the only contribution on $u$ comes from projections of $u$ on the solutions of the Robin eigenvalue problem, for $\mu_1$.  

There is no guarantee for uniqueness (up to constant multiplications) for the solution of problem ${\mathcal L} u = \mu_1 u$, subject to Robin boundary conditions. However, using results from the Fredholm theory (for the inverse of $\mathcal L$) we know that $\mu_1$ has finite multiplicity, hence there exists $m$ solutions of problem $\mathcal{L} u = \mu_1 u$, subject to Robin boundary conditions, $\{ \psi_{1}^{(1)}, \cdots, \psi_{1}^{(m)}\}$.   

Taking into account that $d_{k}=(u, \psi_{k})_{L^2(D)}=0$ for all $k >1$ and the finite multiplicity of the first eigenfunction we conclude that $u= \sum_{\ell=1}^{m} (u, \psi_{1}^{(\ell)} )_{L^2(D)} \psi_{1}^{(\ell)}$ , with $\mathcal L \psi_{1}^{(\ell)} = \mu_1 \psi_{1}^{(\ell)}$, $\ell =1, \cdots, m$, with Robin boundary conditions. Then, 
\begin{eqnarray*}
{\mathcal L} u &&= {\mathcal L}\left(\sum_{\ell=1}^{m} (u, \psi_{1}^{(\ell)} )_{L^2(D)} \psi_{1}^{(\ell)}\right) = \sum_{\ell=1}^{m} (u, \psi_{1}^{(\ell)} )_{L^2(D)} {\mathcal L} \psi_{1}^{(\ell)} \\&&= \mu_1 \left(\sum_{\ell=1}^{m} (u, \psi_{1}^{(\ell)} )_{L^2(D)} \psi_{1}^{(\ell)}\right) =\mu_1 u,
\end{eqnarray*}
and with a similar calculation ${\mathcal N} u + \beta u =0$. Hence, \eqref{E2} implies \eqref{E1} and the proof of this step is complete.

Step 5.  We now show that if $u$ is an eigenfunction of $\mathcal{L}$ (with Robin boundary conditions and $\| u\|_{L^2(D)}=1$)  with eigenvalue $\mu_1$, then so it is $|u|$.

To show that note that since $u$ is an eigenfunction of ${\mathcal L}$  we have by \eqref{Va} that
\begin{eqnarray}\label{G}
Q_{\beta}(u,u)=\mu_1= \min \{ Q_{\beta}(w,w) \,\, \mid \,\, \| w\|_{L^2(D)}=1\}.
\end{eqnarray}

We recall the elementary inequality $| |a|-|b| | \le |a -b|$ which we apply for $a=u(x)$, $b=u(y)$ to see that ${\mathcal E}(|u|,|u|) \le {\mathcal E}(u,u)$, hence, 
\begin{eqnarray}\label{HG} 
Q_{\beta}(|u|,|u|) \le Q_{\beta}(u,u)=\mu_1,
\end{eqnarray}
where we also used \eqref{G}. 

But since $\| u\|_{L^2(D)}=1$ also implies that $\| |u| \|_{L^2(D)}=1$, once more by \eqref{G} and \eqref{HG}
 we see that
\begin{eqnarray}
Q_{\beta}(|u|, |u|) =\mu_1,
\end{eqnarray}
which by the results of step 4 yields that $|u|$ is also an eigenfunction.

We note that by the same argument, and since $u^{+}=\frac{1}{2}( |u| + u)$, by a convexity argument the same holds for $u^{+}$ , and subsequently for $u^{-}$. 

Step 6. By the results of step 5 we have that  $\mathcal{L} |u| = \mu_1 |u|$ in $D$, which in turn implies that $\mathcal{L} |u| \ge 0$ in $D$, with $|u| \ge 0$.  With the comment above, same applies for $u^{+}$.  By the strong maximum principle  (see Theorem 1.1, or Theorem 1.2 in  \cite{jarohs2019strong}) this implies that either $u^{+}$ is strictly positive or identically zero.  Exactly the same applies for 	$|u|$. Since having $u^{+}$ and $|u|$ identically zero will lead to the trivial solution, we conclude that $u^{+}>0$ hence $u>0$.

\end{proof}

\bibliography{sn-bibliography}


\end{document}